\documentclass{amsart}
\usepackage{ucs}
\usepackage[utf8x]{inputenc}
\usepackage{comment}
\usepackage{xcolor}

\usepackage{amsmath,amssymb,amsfonts,amstext,amsthm}
\usepackage{relsize}
\usepackage{url}
\usepackage{graphicx}
\usepackage{tikz-cd}
\input xy
\xyoption{all}
\usepackage[all]{xy}
\usepackage{hyperref}
\usepackage{graphicx}
\usepackage{stmaryrd}

\usepackage{longtable}

\usepackage{subcaption}

\usepackage{cleveref}
\crefname{subsection}{§}{§§}
\Crefname{subsection}{§}{§§}

\graphicspath{{images/}{../images/}}
\usepackage{subfiles}

\newcommand{\nwc}{\newcommand}

\nwc{\aaa}{\mathcal{F}}
\nwc{\aap}{\mathcal{F}_{P}}
\nwc{\al}{\alpha}

\nwc{\C}{\mathbb{C}}
\nwc{\cb}{\overline{C}}
\nwc{\ccc}{\mathfrak{c}}
\nwc{\ch}{\widehat{C}}
\nwc{\cin}{\textbf{(v)}}
\nwc{\cl}{C'}
\nwc{\cp}{\mathcal{C}_{P}}
\nwc{\cpll}{\mathfrak{c}_{P'}}
\nwc{\ct}{\widetilde{C}}

\nwc{\dd}{\mathcal{L}}
\nwc{\ddd}{\mathfrak{d}}
\nwc{\ddl}{\mathcal{L}'}
\nwc{\dlp}{\delta_{P}}
\nwc{\doi}{\textbf{(ii)}}

\nwc{\fl}{\flushleft}
\nwc{\fff}{\mathcal{F}}
\nwc{\ffp}{\mathcal{F}_{P}}
\nwc{\ffq}{\mathcal{F}_{Q}}
\nwc{\ffl}{\mathcal{F}'}

\nwc{\G}{\mathcal{G}}
\nwc{\Ga}{\Gamma}
\nwc{\gtl}{\widetilde{g}}

\nwc{\hra}{\hookrightarrow}
\nwc{\hua}{h^{1}(C,\aaa )}

\nwc{\kk}{{\rm K}}

\nwc{\llb}{\mathcal{L}}

\nwc{\mb}{\mathbb}
\nwc{\mc}{\mathcal}
\nwc{\mm}{\mathfrak{m}}
\nwc{\mmp}{\mathfrak{m}_{P}}
\nwc{\mpd}{\mathfrak{m}_{P}^{2}}

\nwc{\nn}{\mathbb{N}}

\nwc{\ob}{\overline{\mathcal{O}}}
\nwc{\obr}{\mathcal{O}^*}
\nwc{\obp}{\overline{\mathcal{O}}_P}
\nwc{\och}{\mathcal{O}_{\hat{C}}}
\nwc{\oh}{\hat{\mathcal{O}}}
\nwc{\ohp}{\hat{\mathcal{O}}_{P}}
\nwc{\ol}{\mathcal{O}'}
\nwc{\oma}{\Omega (\mathfrak{a})}
\nwc{\omo}{\Omega (\mathcal{O})}
\nwc{\oo}{\mathcal{O}}
\nwc{\op}{\mathcal{O}_P}
\nwc{\opc}{\mathcal{O}_{P,C}}
\nwc{\oph}{\hat{\mathcal{O}}_{P}}
\nwc{\opl}{\mathcal{O}_{P}'}
\nwc{\oplc}{\mathcal{O}_{P,C}'}
\nwc{\opll}{\mathcal{O}_{P'}}
\nwc{\opt}{\tilde{\mathcal{O}}_{P}}
\nwc{\optt}{{\mathcal{O}}_{\tilde{P}}}
\nwc{\oq}{\mathcal{O}_{Q}}
\nwc{\oqt}{\tilde{\mathcal{O}}_{Q}}
\nwc{\ot}{\widetilde{\mathcal{O}}}
\nwc{\overop}{\bar{\oo}_{P}}

\nwc{\pb}{\overline{P}}
\nwc{\pbb}{P^*}
\nwc{\pbi}{\overline{P_{i}}}
\nwc{\pbr}{\overline{P_{r}}}
\nwc{\pgmd}{\mathbb{P}^{g+2}}
\nwc{\pgmu}{\mathbb{P}^{g+1}}
\nwc{\ph}{\hat{P}}
\nwc{\pp}{\mathbb{P}}
\nwc{\prv}{\noindent\textbook{Proof}:}
\nwc{\pt}{\widetilde{P}}
\nwc{\ptl}{\tilde{P}}
\nwc{\pum}{\mathbb{P}^{1}}

\nwc{\qh}{\hat{Q}}
\nwc{\qtl}{\tilde{Q}}
\nwc{\qua}{\textbf{(iv)}}

\nwc{\ra}{\rightarrow}
\nwc{\rh}{\hat{R}}

\nwc{\sei}{\textbf{(vi)}}
\nwc{\sep}{\beq\ast\ \ast\ \ast\enq}
\nwc{\sig}{\sigma}
\nwc{\Sig}{\Sigma}
\nwc{\ssp}{S_{P}}
\nwc{\sss}{{\rm S}}

\nwc{\tre}{\textbf{(iii)}}

\nwc{\um}{\textbf{(i)}}

\nwc{\vpb}{v_{\overline{P}}}
\nwc{\vtxp}{\widetilde{V}_{x,P}}
\nwc{\vxp}{V_{x,P}}

\let \wt=\widetilde
\let \mc=\mathcal

\nwc{\wh}{\hat{\omega}}
\nwc{\whp}{\hat{\omega}_{P}}
\nwc{\woch}{\omega\cdot\mathcal{O}_{\hat{C}}}
\nwc{\woh}{\omega\cdot\hat{\mathcal{O}}}
\nwc{\ww}{\omega}
\nwc{\wwb}{\omega^*}
\nwc{\wwct}{\omega _{\widetilde{C}}}
\nwc{\wwh}{\widehat{\omega}}
\nwc{\wwhp}{\widehat{\omega}_P}
\nwc{\wwp}{\omega _{P}}
\nwc{\wwt}{\widetilde{\omega}}
\nwc{\wwtp}{\widetilde{\omega}_P}

\nwc{\zz}{\mathbb{Z}}

\newtheorem{coro}{Corollary}[section]

\newtheorem{dfn}[coro]{Definition}

\newtheorem{lemma}[coro]{Lemma}

\newtheorem{prop}[coro]{Proposition}

\newtheorem{rem}[coro]{Remark}

\newtheorem{thm}[coro]{Theorem}
\newtheorem{conj}[coro]{Conjecture}
\newtheorem{ex}[coro]{Example}
\let \fl=\flushleft

\let \ga=\gamma
\let \sub=\subset
\let \be=\beta
\let \al=\alpha
\let \pr=\prime
\let \la=\lambda

\let \de=\delta

\let \ov=\overline
\let \De=\Delta

\def\<{\langle}
\def\>{\rangle}

\def\B{\mathcal B}
\def\ord{ord }

\def\Z{\mathbb Z}

\def\C{\mathbb C}
\def\P{\mathbb P}

\def\Aut{\mbox{Aut}}

\def\D{\Delta}

\usepackage[margin=3cm]{geometry}
\usepackage{xy}
\usepackage{graphics}

\title{Arithmetic inflection of superelliptic curves}
\author{Ethan Cotterill}

\address{IMECC, Unicamp, Rua S\'ergio Buarque de Holanda, 651, 13.083-859 Campinas SP, Brazil}

\email{cotterill.ethan@gmail.com}

\author{Ignacio Darago}

\address{Dept of Mathematics, University of Chicago, 5734 S. University Avenue, Chicago, IL 60637, USA}

\email{idarago@math.uchicago.edu}

\author{Cristhian Garay L\'opez}

\address{CIMAT,
Jalisco S/N, Col. Valenciana CP. 36023 Guanajuato, Gto, M\'exico
}

\email{cristhian.garay@cimat.mx}

\author{Changho Han}

\address{Dept of Mathematics, Korea University, 145 Anam-ro, Seongbuk-gu, Seoul 02841, Republic of Korea}

\email{changho\_han@korea.ac.kr}

\author{Tony Shaska}

\address{Dept of Mathematics, Oakland University, Rochester, MI, 48309, USA}

\email{shaska@oakland.edu}

\begin{document}

\begin{abstract}
In this paper, we explore the inflectionary behavior of linear series on {\it superelliptic} curves $X$ over fields of arbitrary characteristic. Here we give a precise description of the inflection of linear series over the ramification locus of the superelliptic projection; and we initiate a study of those {\it inflectionary varieties} that parameterize the inflection points of linear series on $X$ supported away from the superelliptic ramification locus that is predicated on the behavior of their Newton polytopes.
\end{abstract}

\maketitle

\section{Beyond arithmetic inflection of hyperelliptic curves}
In the study of linear series on complex algebraic curves, a foundational role is played by {\it Pl\"ucker's formula}, which expresses the global inflection of a linear series $g^r_d$ in terms of the projective invariants $(d,r)$ and the genus $g$ of the underlying curve $X$. It is natural to ask for analogues of Pl\"ucker's formula over base fields $F$ other than $\mb{C}$. Inflection, both local and global, then depends on information that refines the numerical data $(d,g,r)$; for example, when $F=\mb{R}$, the number of real inflection points of a (real) linear series on a (real) curve $X$ depends on the topology of the real locus $X(\mb{R})$. 

\medskip
In the papers \cite{BCG,CDH,CG1,CG2}, the first, second, third, and fourth authors studied $F$-rational inflectionary loci for certain linear series on {\it hyperelliptic} curves $X$ defined over $F$. Whenever $\text{char}(F) \neq 2$ and a hyperelliptic curve $X$ has an $F$-rational ramification point $\infty_X$, $X$ admits an affine model $y^2=f(x)$ in ambient coordinates $x$ and $y$ with respect to which 
the complete series $|\ell \infty_X|$ has a distinguished basis of {\it monomials} in $x$ and $y$. The inflection of $|\ell \infty_X|$ and those subseries corresponding to truncations of the distinguished monomial basis comprise determinantal loci cut out by the determinants of Wronskian matrices whose entries are Hasse derivatives.
Somewhat surprisingly, Hasse Wronskians helped clarify both the column-reduction of Wronskian matrices in calculating the inflection of linear series over the hyperelliptic ramification locus, and the structure of {\it inflection polynomials} whose roots parameterize the $x$-coordinates of $\ov{F}$-inflection points of linear series over the complement of the hyperelliptic ramification locus. 
    
\medskip
The aim of this paper is to extend our local analysis of inflection in the hyperelliptic setting to {\it superelliptic} curves. These are cyclic covers of $\mb{P}^1$; whenever the degree of the cyclic cover is not divisible by $\text{char}(F)$, such a cover is defined by an affine equation $y^n=f(x)$ with $n \geq 2$. Superelliptic curves retain many of the salient features that make the projective geometry of hyperelliptic curves accessible. Crucially, complete linear series determined by multiples of a superelliptic $F$-rational ramification point have a basis of monomials in $x$ and $y$ that naturally generalizes the monomial basis operative in the hyperelliptic context. 

\medskip
Broadly speaking, this paper is in two parts. In the first, spanned by sections 2 and 3, we develop the basic combinatorial machinery required to compute inflectionary loci associated with linear series on a (fixed) superelliptic curve. There is a natural dichotomy between inflection that is supported along the ramification locus $R_{\pi}$ of the cyclic cover, and that which is not. Roughly speaking, inflection that arises from $R_{\pi}$ is independent of the underlying superelliptic curve, whereas inflection away from $R_{\pi}$ depends on the point in moduli we choose.

\medskip
The main result of the first half of this paper describes the local geometry of the Hasse Wronskian 
in a point of the superelliptic ramification locus $R_{\pi}$ along a fixed superelliptic curve.

\begin{thm}(Theorem~\ref{Thm 3.9 generalized}) Assume that $\ell \geq 2g+n-1$, $\ell=n\al$ and $d= n\be+1$, where $\al$ and $\be$ are positive integers for which $\frac{\al}{\be}>n-1$. For any field $F$ of characteristic that is either zero or very large, the lowest $y$-adically valued term of the Hasse Wronskian determinant that computes the inflection of $\ell \infty_X$ in a superelliptic ramification point $(\ga,0) \in R_{\pi} \setminus \{\infty\}$ is simultaneously equal to:
\begin{enumerate}
\item
$(\prod D^{\mu_i}_yb_i)|_{(\gamma,0)} \cdot \det N(n,g,\ell) \cdot y^{\mu(B)}$, where $\mu_i=v_y(b_i)$, $N(n,g,\ell)=(\binom{\mu_j}{i})_{0 \leq i,j \leq \ell-g}$, and
\[
\mu(\mc{B})=\frac{(n-1)n^2(n+1)}{24}\be^2+ \frac{(n-1)n(5-n)}{12}\be .
\]
\item 
The lowest $y$-adically valued term of
\[
(D_y^n (x-\ga))^{n\binom{\al-(n-1)\be}{2}} \sum_{p \in \mc{P}^{\ast}}  \det M(p)
\]
where $D^n_y$ denotes the $n$-th Hasse dervative with respect to $y$, $\mc{P}^{\ast}$ is the product of \emph{Pl\"ucker posets} (defined in $\S 3.2$) corresponding to the columns of the Hasse Wronskian matrix, and $M(p)$ is a matrix of monomials in the derivatives $D_y^n (x-\ga)$.
\end{enumerate}
\end{thm}

\medskip
Away from $R_{\pi}$, on the other hand, the inflection of subseries of $|\ell \infty_X|$ is controlled by superelliptic {\it inflection polynomials}, whose roots parameterize the $x$-coordinates of $\ov{F}$-inflection points exactly as in the hyperelliptic case. 
When the superelliptic curve is allowed to vary in a family, the zero loci of corresponding inflection polynomials are varieties of dimension equal to that of the underlying family.
The second half of this paper, spanned by sections 4 and 5, is a close study of the geometry of these {\it inflectionary varieties} when either i) the family of superelliptic curves is a superelliptic analogue of a Legendre or Weierstrass pencil (in a very precise sense) of elliptic curves, or ii) the underlying family of curves is the two-dimensional family of {\it bielliptic} curves in genus two or a special subpencil thereof.
More precisely, we focus on {\it atomic} inflection polynomials $P^{\ell}_m(x)$, whose zeroes are those of the $m$-th Hasse derivative with respect to $x$ of $y^{\ell}$. Inflection polynomials in general 
are determinants in atomic inflection polynomials.

\medskip
Our  main results in the second half of the paper are about Newton polytopes of atomic inflection polynomials, which shed light on the singularities of the inflectionary varieties they define.

\begin{thm} (Theorems~\ref{generic_legendre_newton_polygon},\ref{generic_support_legendre_newton_polygon}, \ref{support_legendre_newton_polygon_z=2}, \ref{weierstrass_Newton_polygon}, \ref{prop:D4_newton_polygon_one}) Suppose that $\text{char}(F)$ is either zero or sufficiently positive.
\begin{enumerate}
\item[i)] Given positive integers $a$, $b$, $c$, $\ell$, $m$ and $n$ as above, the Newton polygon of the inflection polynomial $P^{\ell}_m$ derived from the superelliptic Legendre family $y^n=x^a(x-1)^b(x-\la)^c$ is
\[
{\rm New}(P^{\ell}_m)= \text{Conv}((ma+mc-m,0), (ma+mb+mc-m,0), (ma-m,mc), (ma+mb-m,mc))
\]
whenever $n>(a+b+c)\ell$. In this situation, the Newton polygon is generic, i.e., it has maximal support. On the other hand, when $a=b=c=1$, $n=2\ell$, and $m \geq 2$, the Newton polygon is
\[
\text{New}(P^{\ell}_m)= \text{Conv}((0,m),(m-2,m),(m-2,2),(2m-1,1),(2m-1,0),(2m,0)).
\]
\item[ii)] Suppose that $n=2\ell$. For every positive integer $m \geq 3$, the Newton polygon of the inflection polynomial $P^{\ell}_m$ derived from $y^n=x^3+ \la x+ 2$ with respect to affine coordinates centered in $(x=1,\la=-3)$ is
\[
{\rm New}(P^{\ell}_m)=\text{Conv}((0, \lceil m/2 \rceil),(0,m),\de_{2|(m-1)}(1,(m-1)/2), (m-2,1),(2m-1,0),(2m,0))
\]
in which $\de_{2|(m-1)}$ indicates that this vertex is only operative when $m$ is odd.
\item[iii)] Suppose that $u=\frac{\ell}{n}$ is neither an integer multiple of $\frac{1}{3}$ nor of $\frac{1}{5}$. Given a positive integer $m \geq 2$, let $\mathcal{C}^{\ell}_m= (P^{\ell}_m(x,s)=0)$ denote the $m$-th inflectionary curve derived from the \emph{$D_4$ pencil} $y^2=x^5+x^3+sx$, $s \in F$. Its Newton polygon ${\rm New}(\mc{C}^{\ell}_m)$ is the lattice simplex with vertices $(0,m)$, $(2m,0)$ and $(4m,0)$.
\end{enumerate}
\end{thm}

\subsection{Roadmap}
A detailed synopsis of the material following this introduction is as follows. In Section~\ref{sec:superelliptic_curves}, we introduce superelliptic curves and their linear series. Lemma~\ref{monomial_basis} characterizes the monomial basis for the complete series associated with an arbitrary sufficiently large multiple of a superelliptic ramification point. In Section~\ref{sec:global_and_local_inflection_formulae}, we begin our quantitative study of inflection of linear series on superelliptic curves in earnest. Theorem~\ref{thm:global_A1_formula} establishes that whenever appropriate numerological conditions are satisfied\footnote{These are the conditions that ensure that the jet bundle that computes inflection is {\it relatively orientable}.}, a well-defined {\it $\mb{A}^1$-inflection class} exists in the Grothendieck--Witt group of the base field $F$. Just as in the hyperelliptic case worked out in \cite{CDH}, the {\it global} $\mb{A}^1$-class of the inflectionary locus of a linear series on a superelliptic curve is less interesting than its individual {\it local} inflectionary indices. 

\medskip
In the present paper, we have not carried out the full calculation of local inflectionary indices in $\mb{A}^1$-homotopy theory. 
We have, however, deepened the local analysis of inflectionary indices in other ways. In Section~\ref{sec:arithmetic_inflection}, we prove Theorem~\ref{Thm 3.9 generalized}, which characterizes the lowest-ordest terms of those Hasse Wronskians associated to complete linear series $|\ell \infty_X|$ in superelliptic ramification points. 
Section~\ref{sec:Hasse_inflection_polynomials} introduces {\it Hasse inflection polynomials}, which parameterize the inflection of subseries of $|\ell \infty_X|$ away from the superelliptic ramification locus. The characteristic recursion that {\it atomic} inflection polynomials satisfy is spelled out in Proposition~\ref{prop:infl_recursion}. Closely-related polynomials have been studied before, notably by Towse \cite{To}, who used their analogues constructed using usual derivatives to study the inflection of superelliptic canonical series. The main novelty in our approach is to put these to use in studying the variation of inflection points in families of marked superelliptic curves. For families of index-$n$ superelliptic curves defined over a {\it ring} $R$, our Hasse inflection polynomials are defined over $R[\frac{1}{n}]$; see Remark~\ref{rmk:infl_recursion_ext}.

\medskip
Section~\ref{sec:legendre_and_weierstrass_pencils} is a close study of the (atomic) inflectionary curves cut out by superelliptic analogues of Legendre and Weierstrass pencils of elliptic curves. In general, the singularities of fibers of a family will contribute ``extra" inflection; so it is natural to expect that these manifest as singularities in the corresponding inflectionary varieties. Legendre inflectionary pencils derived from presentations $y^n=x^a(x-1)^b(x-\la)^c$ with $a,b,c \in \mb{N}$ are the focus of Section~\ref{sec:symmetries_of_Legendre_curves}. These were previously studied by the first four authors when $n=2$ and $a=b=c=1$; here we extend the earlier analysis in a couple of distinct directions. Theorem~\ref{superell_inflection_poly_symmetry} establishes that Legendre inflectionary curves inherit automorphisms from their underlying pencils whenever $a=b=c$. Turning to singularities of Legendre inflectionary curves $\mc{C}^{\ell}_m$ inherited from the underlying pencils, we then prove Theorem~\ref{generic_legendre_newton_polygon}, which gives a generic expectation for the Newton polygon $\text{New}(\mc{C}^{\ell}_m)$ of the $m$-th Legendre inflectionary curve with respect to coordinates centered in the origin where $\mc{C}^{\ell}_m$ is singular. Whenever $\text{char}(F)$ is either zero or sufficiently large, Theorem~\ref{generic_support_legendre_newton_polygon} establishes that the generic expectation is met whenever the parameter $u=\frac{\ell}{n}$ is itself ``generic" (and in particular, whenever $u$ is sufficiently large relative to $a$, $b$, and $c$); while Theorem~\ref{support_legendre_newton_polygon_z=2} describes $\text{New}(\mc{C}^{\ell}_m)$ when $u=\frac{1}{2}$, which is a value of particular significance insofar as it includes the (unique) hyperelliptic case in which $\ell=1$ and $n=2$.

\medskip
Our explicit identification of Newton polygons of singularities of atomic inflectionary curves is predicated on the recursive structure of the associated atomic inflection polynomials. In particular, the coefficients of these in terms corresponding to vertices of Newton polygons tend to split {\it $u$-linearly} over $F$. We push this principle further in Section~\ref{sec:weierstrass_inflectionary_curves}, in which we study Weierstrass inflectionary curves derived from presentations $y^n=x^3+\la x+2$. Here we assume for simplicity that $u=\frac{1}{2}$, though a number of our arguments are nonspecific to this case.
Viewed as an affine curve $\mc{C}^{\ell}_m \sub \mb{A}^2_{x,\la}$, each Weierstrass inflectionary curve comes equipped with a cyclic $\mu_3$-action, which permutes its distinguished singularities in $(\zeta^j,-3\zeta^{-j})$, $j=0,1,2$ inherited from the underlying pencil; see Theorem~\ref{weierstrass_infl_curve_autom}. In Theorem~\ref{weierstrass_Newton_polygon} we compute the Newton polygon of $\mc{C}^{\ell}_m$ in coordinates adapted to the singular point $(1,3)$; while in Conjecture~\ref{inner_edge_conjecture} we predict the exact normal form of the corresponding singularity. This, in turn, leads to Conjecture~\ref{conj:separability}, which predicts that each of these singularities is {\it Newton non-degenerate}, and we present some experimental evidence in favor of this. Newton non-degeneracy would imply, in particular, that $\mc{C}^{\ell}_m$ has multiple-point singularities with smooth branches in $(\zeta^j,-3\zeta^{-j})$ whenever $m \geq 3$. Our Newton polygon calculation also immediately (and unconditionally) yields the $\delta$-invariant of each of the three distinguished singularities; assuming $\mc{C}^{\ell}_m$ has no further singularities and is irreducible, this in turn leads to an explicit expectation for the geometric genus of $\mc{C}^{\ell}_m$; see Conjectures~\ref{conj:weierstrass_sing} and \ref{conj:weierstrass_geom_genus}, respectively. It is natural to wonder what shapes our results (and in particular, Newton polygons) for $\mc{C}^{\ell}_m$ might take when $\text{char}(F)$ is positive and small relative to $m$. Remark~\ref{rem:weierstrass_behavior_in_low_char} addresses the $p$-adic valuations of (some of) the hypergeometric functions that arise as coefficients of inflectionary Newton polygons; while Propositions~\ref{prop:Cell3_quotient_one} and \ref{prop:Cell3_quotient_two} together give a complete topological description of the $\mu_3$-quotient of $\mc{C}^{\ell}_3$ in arbitrary odd characteristic.

\medskip
In Section~\ref{sec:bielliptic}, we investigate  superelliptic inflectionary varieties derived from bielliptic curves in $\mc{M}_2$, 
especially curves with automorphism groups isomorphic to either of the dihedral groups $D_4$ or $D_6$. Over a perfect field $F$ not of characteristic 2 or 3, any such curve is $\ov{F}$-isomorphic to a curve with affine equation $y^2= x^5+ x^3 + sx$ or $y^2=x^6+x^3+ z$, where $s$ and $z$ are the respective modular parameters; and by replacing $y^2$ by $y^n$ we obtain superelliptic analogues in either case. In the $D_4$ case, the inflectionary curves $
\mc{C}_m=\mc{C}^{\ell}_m \sub \mb{A}^2_{x,s}$ always has a singularity in the origin, and in Theorem~\ref{prop:D4_newton_polygon_one} we compute the corresponding Newton polygons, assuming that $u$ is not a multiple of either $\frac{1}{3}$ or $\frac{1}{5}$. We then specialize to the case in which $u=\frac{1}{2}$, and $\text{char}(F)$ is either zero or sufficiently positive. The $x$-discriminant of the $D_4$ pencil vanishes in $s=0$ and $s=\frac{1}{4}$, and the special value $s=\frac{1}{4}$ is associated with singularities of $\mc{C}_m$ supported in $(\pm \sqrt{\frac{-1}{2}},\frac{1}{4})$; these are permuted by a cyclic $\mu_3$-automorphism of $\mc{C}^{\ell}_m$ itself. Conjecture~\ref{conj:D4_inflectionary_singularities} predicts that $\mc{C}_m$ is smooth away from the four distinguished singularities inherited from the $D_4$ pencil; while Conjecture~\ref{conj:D4_newton_polygon_two} gives our expectation for the Newton polygons of $\mc{C}_m$ adapted to either of the singularities in $(\pm \sqrt{\frac{-1}{2}},\frac{1}{4})$ whenever $m \geq 3$.\footnote{We nevertheless anticipate that the proof of Conjecture~\ref{conj:D4_newton_polygon_two} will be straightforward via induction, once the terms of the $D_4$ inflection polynomials corresponding to the vertices of the putative Newton polygons have been explicitly identified.} These, in turn, lead to Conjecture~\ref{D4_geometric_genus}, which gives an explicit prediction for the geometric genus of $\mc{C}_m$ whenever $m \geq 3$. 

\medskip
In Proposition~\ref{prop:D4_sato-tate}, on the other hand, we show that the (renormalized Hasse--Weil deviations of) $\mb{F}_p$-rational points counts on $\mc{C}_2$ as $p$ varies over all primes are equidistributed with respect the Sato--Tate distribution of an elliptic curve without complex multiplication obtained by desingularizing $\mc{C}_2$. In Conjecture~\ref{conj:D_6Newton_polygon}, we make a precise (and rather involved) prediction regarding the singularities and geometric genera of $D_6$ inflectionary curves; while the final Section~\ref{sec:inflectionary_surfaces_from_bielliptics} is a preliminary exploration of the structure of the (reduced) {\it inflectionary discriminant} curves whose points parameterize those points over which the projection of a bielliptic inflectionary {\it surface} $y^2=x^6-s_1 x^4+ s_2 x^2-1$ to the underlying parameter space $\mb{A}^2_{s_1,s_2}$ fails to be \'etale. This will take place above singular curves; so the discriminant $\De$ of the inflectionary surface always comprises a component of the inflectionary discriminant. We show that the reduced structure $\De_*$ on $\De$ is an irreducible rational curve, whose parameterization we compute explicitly. We also describe the ``extra" components of the inflectionary discriminant $\De_{m}^{\ell}$ when $m \in \{3,4,5\}$.
Throughout this paper, Mathematica, Macaulay2, and Sage have played a vital role in both developing our conjectures and proving our theorems.

\subsection{Acknowledgements}
The first author would like to thank Vlad Matei for helpful conversations in the early stages of this project. Thanks are also due to the two anonymous referees, whose comments and suggestions have led to an improved exposition. The third author was funded by CONACYT project no. 299261.

\section{Superelliptic curves}\label{sec:superelliptic_curves}
Superelliptic curves are abelian covers of the projective line with cyclic automorphism groups; see \cite{m-sh} for a comprehensive discussion of these.
We will always assume that our covers are tame. Explicitly, assuming the branch points of a given cover $\pi: X \to \P^1$ comprise pairwise-distinct points $a_{1}, \ldots, a_{r} \in \P^1$, the superelliptic curve $ X$ is a compactification of an affine irreducible algebraic curve with presentation
\begin{equation}\label{ngonal}
y^{n}=\prod_{j=1}^{r}(x-a_{j})^{l_{j}}
\end{equation}
in which $l_{1},\ldots,l_{r} \in \{1,\ldots,n-1\}$ and  $\gcd(n,l_{1},\ldots,l_{r})=1$. The point at infinity is a branch point of $\pi$ if and only if $l_{1}+\cdots+l_{r}$ is not congruent to zero modulo $n$.

\subsection{Linear series on superelliptic curves with reduced branch loci}
In this subsection, we assume that every branching index $l_j$, $j=1,\dots,r$ singled out by the affine presentation \eqref{ngonal} is equal to one, and that $\gcd(n,d)=1$, where $d=\deg(f)=r$.
Let $a_i$, $i=1,\dots,d$ denote the $d$ distinct roots of $f(x)$, and for each $i$, let $b_i=(a_i,0)$ denote the corresponding affine branch point of $\pi:  X\to\P^1$.  %
For any non-branch point $c \in \P^1$, let $P^c_1, \dots, P^c_n$ denote its preimages in $X$; and let $P^{\infty}$ denote the preimage of the point at infinity. 
 The abstract curve $X: y^n=f(x)$ is smooth, and its projection $\pi: X \ra \mb{P}^1$ is totally ramified in $P^{\infty}$; see, e.g., \cite{GPS}. On $X$, we distinguish divisors
\[
\begin{split}
\text{div}(x-c) &= \sum_{j=1}^n  P^c_j - n P^\infty; \\
\text{div}(x-a_i) &=  n b_i -n P^\infty; \\
\text{div}(y) &= \sum_{j=1}^d b_j - d  P^\infty; \text{ and}\\
\text{div}(dx) &= (n-1)\sum_{j=1}^d b_j - \left(n+1\right) P^\infty.
\end{split}
\]
Since $\text{div}(dx)$ is a canonical divisor, it has degree $2g-2$, where $g$ is the 
genus of $X$; and therefore $2g-2=nd-n-d-1$, i.e., 
$g= \frac{(d-1)(n-1)}{2}$. The following lemma will play a crucial role in the sequel.

\begin{lemma}\label{monomial_basis} Let  $n$ and $d$ be as above, and assume that $\gcd(n,d)=1$. For every nonnegative integer $\ell$, a basis of global sections for $\mc{O}(\ell\infty)$ over $F$ is given by 
\[ \B_{\ell;n,d}:=\left\{    x^i y^j \;|    0 \leq i, \;  0 \leq j \leq n-1, \text{ and } ni+dj \leq \ell   \right\}.\]
\end{lemma}

\begin{proof}
The pole orders of $x$ and $y$ at infinity are $n$ and $d$, respectively, so by additivity the pole order at infinity of any given monomial $x^i y^j$ is $\ord_\infty x^i y^j = ni + dj$. Because $\gcd(n,d)=1$, values of these linear combinations are pairwise distinct.
\end{proof}

\begin{rem}
\emph{Whenever $\ell \infty$ is linearly equivalent to the pullback of a divisor $D$ on an ambient smooth toric surface $S$ containing $X$, inflection of the linear series $|\mc{O}(\ell\infty)|$ on $X$ may be re-interpreted purely in terms of the geometry of $S$. Indeed, geometrically $p \in X$ is an inflection point of $|\mc{O}(\ell\infty)|$ if and only if the unique osculating hyperplane has contact order at least $e+1$, where $e$ is the projective rank of $|\mc{O}(\ell\infty)|$. But whenever the morphism $\varphi$ defined by $|\mc{O}(\ell\infty)|$ factors through $S$, the osculating hyperplane in the target of $\varphi$ pulls back to an {\it extactic} curve on $S$ in the sense of Cayley. In this situation, $p \in X$ is an inflection point of $|\mc{O}(\ell\infty)|$ whenever there is a curve of class $D$ that intersects $X$ with contact order at least $e+1$ in $p$.}
\end{rem}

\section{Global and local superelliptic inflection formulae}\label{sec:global_and_local_inflection_formulae}

\subsection{A global inflection formula}
We begin by giving a superelliptic analogue of the global $\mb{A}^1$-Pl\"ucker formula for arbitrary multiples of a $g^1_2$ on a hyperelliptic curve \cite[Thm. 3.1]{CDH}.

\begin{thm}\emph{(Generalization of \cite[Thm. 3.1]{CDH})}\label{thm:global_A1_formula}
Let $X$ denote a cyclic $n$-fold cover of $\mb{P}^1$ defined over a field $F$ of characteristic relatively prime to $n$, $n \geq 2$. Assume that the superelliptic curve $X$ has an $F$-rational point $\infty_X$, over which the associated superelliptic projection $\pi : X \to \mathbb{P}^1$ is ramified. For every positive integer $\ell$, the complete linear series $|\ell \infty_X|$ has a well-defined arithmetic $\mb{A}^1$-inflection class in $\text{GW}(F)$ given by $\frac{\ga_{\mb{C}}}{2} \mb{H}$ whenever either $\ell$ or the dimension of $|\ell \infty_X|$ as a vector space is even. Here $\ga_{\mb{C}}$ denotes the $\mb{C}$-inflectionary degree computed by Pl\"ucker's formula, and $\mb{H}=\langle 1 \rangle+ \langle -1 \rangle$ denotes the hyperbolic class.
\end{thm}

\begin{proof}
Exactly as in \cite{CDH}, the existence of the $\mb{A}^1$-inflection class is guaranteed provided the line bundle $L^{\otimes (r+1)} \otimes K_X^{\otimes \binom{r+1}{2}}$ is of even degree, where $L$ and $r$ denote the line bundle and the projective dimension of the complete linear series $|\ell \infty_X|$, respectively.
\end{proof}

\subsection{Arithmetic inflection of linear series on superelliptic curves}\label{sec:arithmetic_inflection}

Just as in \cite{CDH}, {\it local} inflection formulae are significantly more interesting than their global aggregates. Local inflection indices are computed by Wronskian determinants; for an elementary account of how this works over $\mb{C}$, see \cite[\S 4]{SS}. Since we work in arbitrary characteristic, our Wronskians are {\it Hasse Wronskians} built out of Hasse derivatives. A basic principle that holds in arbitrary characteristic is that ramification points of the superelliptic projection $\pi: X \ra \mb{P}^1$ are nontrivially inflected for linear series on $X$. In this subsection, we will produce an explicit description of the lowest-order terms of Hasse Wronskian determinants over the ramification locus $R_{\pi}$. 

\medskip
To state the main result of this section, which generalizes \cite[Thm. 3.9]{CDH} to the superelliptic context, we will make use of {\it Pl\"ucker posets}. Given a Grassmannian $G=G(k,n)$, the {\it Pl\"ucker poset} of $G$ is the partially ordered set of partitions that fit inside a $k \times (n-k)$ rectangle. A {\it path} in a Pl\"ucker poset is any sequence of partitions, ordered from smallest to largest, such that the weights increase one by one. Paths in Pl\"ucker posets form the basis of a convenient indexing scheme for lowest-order monomials in Hasse Wronskian determinants. 

\medskip
Accordingly, assume that $\ell \geq 2g+n-1$, $\ell=n\al$ and $d= n\be+1$, where $\al$ and $\be$ are positive integers for which $\frac{\al}{\be}>n-1$; and assume that $(\ga,0)$ is a ramification point of the superelliptic projection not lying over $\infty$. As in \cite[Thm. 3.9]{CDH}, there is an {\it inflectionary basis} of generalized monomials $(x-\ga)^i y^j$ adapted to $(\ga,0)$ (and as in \cite[proof of Thm. 3.9]{CDH}, the corresponding Hasse Wronskians are independent of $\ga$); then $y$ is a uniformizer of $\mc{O}_{X,(\gamma,0)}$. We now order the elements of $\mc{B}$ according to their $y$-adic valuations $v_y$.
Given $0 \leq i_0 \leq \al$, let $\mc{B}^{(i_0)} \sub \mc{B}$ denote the subset comprising monomials of the form $x^{i_0} y^j$ for some $j$. Clearly, $\mc{B}= \bigsqcup_{i_0=0}^{\al} \mc{B}^{(i_0)}$, and moreover we have $\mc{B}^{(i_0)} < \mc{B}^{(j_0)}$ whenever $i_0<j_0$, by which we mean that the $y$-adic valuation of any element of $\mc{B}^{(i_0)}$ is less than the $y$-adic valuation of any element of $\mc{B}^{(j_0)}$. On the other hand, the fact that $v_y(x^{i_0} y^j)< v_y(x^{i_0} y^k)$ whenever $j<k$ describes the $y$-adic total order on $\mc{B}^{(i_0)}$. Let $\mu_i:= v_y(b_i)$, $i=0,\dots,\ell-g$ denote the {\it inflectionary orders} of the elements $b_i$ of the monomial basis $\mc{B}$, ordered $y$-adically as above.

\begin{thm}\label{Thm 3.9 generalized}\emph{(Generalization of \cite[Thm. 3.9]{CDH})} Assume that $\ell \geq 2g+n-1$, $\ell=n\al$ and $d= n\be+1$, where $\al$ and $\be$ are positive integers for which $\frac{\al}{\be}>n-1$. For any field $F$ of characteristic that is either zero or sufficiently large, the lowest $y$-adically valued term of the Hasse Wronskian determinant $w(\mc{B})$ associated to the 
inflectionary basis $\mc{B}=\mc{B}_{\ell; n,d}=\{b_i\}_{0 \le i \le \ell -g}$ of Lemma~\ref{monomial_basis} in a superelliptic ramification point $(\ga,0) \in R_{\pi} \setminus \{\infty\}$ is simultaneously equal to:
\begin{enumerate}
\item
$(\prod D^{\mu_i}_yb_i)|_{(\gamma,0)} \cdot \det N(n,g,\ell) \cdot y^{\mu(B)}$, where $\mu_i=v_y(b_i)$, $N(n,g,\ell)=(\binom{\mu_j}{i})_{0 \leq i,j \leq \ell-g}$, and
\[
\mu(\mc{B})=\frac{(n-1)n^2(n+1)}{24}\be^2+ \frac{(n-1)n(5-n)}{12}\be .
\]
\item 
The lowest $y$-adically valued term of
\[
(D_y^n (x-\ga))^{n\binom{\al-(n-1)\be}{2}} \sum_{p \in \mc{P}^{\ast}}  \det M(p)
\]
where $\mc{P}^{\ast}$ is the product of Pl\"ucker posets corresponding to the columns of the Hasse Wronskian matrix $W(\mc{B})$, and $M(p)$ is a matrix of monomials in the derivatives $D_y^n (x-\ga)$, with suitably-renormalized multinomial coefficients, canonically specified by $p \in \mc{P}^{\ast}$.
\end{enumerate}
\end{thm}

\begin{rem}\label{rmk:col_red}
\emph{
    To prove Theorem~\ref{Thm 3.9 generalized}, we use two distinct decompositions of the Hasse Wronskian matrix $W(\mc{B})$. Decomposing each column vector as a linear combination of column vectors of Hasse derivatives of monomial powers of $y$ yields item 1; 
    while decomposing each column vector of $W(\mc{B})$ as a linear combination of column vectors of Hasse derivatives of elements of the distinguished basis $\mc{B}$ and column-reducing using the Fa\`a di Bruno formula yields item 2.}
    
    \medskip
\emph{Comparing the lowest-order terms of the power series expansions of $w(\mc{B})$  in $y$ that result from each of these two decompositions, we obtain a seemingly novel decomposition of a Vandermonde determinant as a linear combination of determinants of matrices $M(p)$ (with evaluating monomials in Hasse derivatives by suitable numbers, see Remark~\ref{rmk:infl_recursion_ext}) coming from a particular product $\mc{P}^*$ of Pl\"ucker posets. This is particularly interesting given that
the $M(p)$ are generalizations of Gessel-Viennot matrices. Indeed, when $n=2$, \cite[Rmk. 3.10]{CDH} establishes that when replacing all monomials in Hasse derivatives by one, $M(p)$ is a Gessel-Viennot matrix; see example~\ref{GV_example_n=2} below.}
\end{rem}

\begin{proof}
With the exception of the explicit identification of the inflectionary multiplicity $\mu(\mc{B})$, the proof of the first item is a standard adaptation of the argument given in the proof of \cite[Lem. 1.2]{EH} using usual derivatives; see also \cite[eq. (2.6)]{Tor} for an argument using Hasse derivatives. Nevertheless, for the sake of completeness we give a proof.

\medskip
Without loss of generality we may (and shall hereafter) suppose that $\ga=0$. A first way to calculate the lowest $y$-adically valued term of
$w(\mc{B})$ involves writing each basis element in $\mc{B}$ as a power series
$b_i= \sum_{k=0}^{\infty} D^k_yb_i|_{(0,0)}\cdot y^k$
in $y$ near the superelliptic ramification point $(0,0)$; and then decomposing each as the sum of its leading term plus higher-order terms. Via multilinearity of the determinant, these power series decompositions induce a decomposition of $w(\mc{B})$; and accordingly 
it suffices to show that 
\[
   w(\{y^{\mu_i}\}_{0 \le i \le \ell -g})=\det N(n,g,\ell) \cdot y^{\sum (\mu_i-i)} \text{ and } \det N(n,g,\ell) \neq 0
\]
and to compute $\sum (\mu_i-i)=\mu(\mc{B})$ explicitly. Note, however, that the determinantal formula in the preceding line follows immediately from \cite[eq. (2.6)]{Tor}; while the fact that $\det N(n,g,\ell) \neq 0$ in $F$ follows from our hypotheses on the characteristic of $F$ and the more general fact that the coefficient of $w(\{y^{\mu_i}\})$ is a nonzero scalar multiple of a Vandermonde determinant \cite[Lem. 1.2]{EH}. 
We defer the computation of $\mu(\mc{B})$ to the proof of the second item.

\medskip
Much as in \cite[proof of Thm. 3.9]{CDH}, the proof of the second item follows from a careful column-reduction of a Wronskian matrix of Hasse $y$-derivatives of the distinguished monomial basis $\mc{B}$ after each of these have been expanded using the Leibniz and Fa\`a di Bruno (chain) rules for Hasse derivatives. More precisely, the latter rules imply that
\begin{equation}\label{derivative_of_x^jy^i}
D^k_y (x^j y^i)=\sum_{\ell=0}^i D^{k-\ell}_y (x^j) \cdot \binom{i}{\ell} y^{i-\ell}
\end{equation}
and
\begin{equation}\label{derivative_of_x^j}
D^k_y x^j= \mathlarger{\sum}_{\substack{\sum_{i=1}^k i c_i=k\\ c_i \geq 0 \text{ for all }i}} \binom{c_1+\cdots+c_k}{c_1,\dots,c_k} 
\binom{j}{c_1+\dots+c_k} x^{j-(c_1+\dots+c_k)} \cdot \mathlarger{\prod}_{i=1}^k (D^i_y x)^{c_i}
\end{equation}
for all nonnegative integers $i$, $j$, and $k$.

\medskip
As in {\it loc. cit.}, 
let $W(\mc{B})$ denote the Wronskian matrix of Hasse $y$-derivatives of elements of the $y$-adically ordered set $\mc{B}$; this is an $(\ell-g+1) \times (\ell-g+1)$ matrix whose $(i,j)$-th entry of $W(\mc{B})$ is equal to the $i$-th derivative of the $j$-th element of $\mc{B}$ with respect to its $y$-adic total order. For every $i_0=0,\dots,\al$, let $W(\mc{B}^{(i_0)})$ denote the submatrix of $W(\mc{B})$ consisting of those columns indexed by $\mc{B}^{(i_0)}$.
We now column-reduce every $W(\mc{B}^{(i_0)})$ using \eqref{derivative_of_x^jy^i}; in doing so, we replace every entry of the form $D^k_y (x^{i_0} y^j)$ by $D^{k-j}_y (x^{i_0})$. Next, we column-reduce each resulting matrix (i.e., each reduction of $W(\mc{B}^{(i_0)})$) using \eqref{derivative_of_x^j}; 
the $k$-th entry of the column of (the reduced version of) $W(\mc{B}^{(i_0)})$ indexed by $x^{i_0} y^j$ becomes
\begin{equation}\label{pure_derivative_columns}
\mathlarger{\sum}_{\substack{\sum_{m=1}^{k-j} m c_m=k-j \\ \sum_{m=1}^{k-j} c_m=i_0}} \binom{i_0}{c_1,\dots,c_{k-j}} \mathlarger{\prod}_{m=1}^{k-j} (D^m_y x)^{c_m}.
\end{equation}

Note that each nonzero product $\mathlarger{\prod}_{m=1}^{k-j} (D^m_y x)^{c_m}$ in \eqref{pure_derivative_columns} is indexed by a partition of $k-j$ with $i_0$ parts, namely $\la=((k-j)^{c_{k-j}},\dots,1^{c_1})$, and that the corresponding coefficient $\binom{i_0}{c_1,\dots,c_{k-j}}$ is a function of $\la$. Here we allow for the possibility that some exponents $c_m$ may be zero. Note, moreover, that $v_y(D^m_y x)=\max(n-m,0)$ whenever $m \le n$, as $(0,0) \in R_{\pi}$. 
On the other hand, clearly $x$ is an infinite formal power series in $y$, 
which implies that whenever $m>n$, we also have $v_y(D^m_y x) \ge 0=\max(n-m,0)$. It follows that
\begin{equation}\label{valuation_of_matrix_entry}
v_y\bigg(\mathlarger{\prod}_{m=1}^{k-j} (D^m_y x)^{c_m}\bigg)= \sum_{m=1}^{k-j} c_m v_y(D^m_y x) \ge \sum_{m=1}^{k-j} c_m \max(n-m,0)
\end{equation}
and that the middle sum in \eqref{valuation_of_matrix_entry} is also a function of the underlying partition $\la$.

\medskip
To go further, we will apply the numerological hypotheses on $\ell$ and $d$ we imposed at the outset to give a more explicit presentation for each of the subsets $\mc{B}^{(i_0)}$, $0 \leq i_0 \leq \al$. The point here is that our basic pole-order condition $n i_0+ dj \leq \ell$ reduces to $i_0 \leq \al- \be j - \frac{j}{n}$. As $0 \leq j \leq n-1$, this is equivalent to requiring that
\begin{equation}\label{pole_order_condition}
i_0 \leq \al- \be j-1 \text{ whenever }j \neq 0.
\end{equation}
Accordingly, let $j=j(i_0)$ denote either the unique positive integer such that $\al- \be(j+1) \leq i_0 \leq \al-\be j-1$ (when $i_0<\al-\be$) or zero (when $\al-\be \le i_0 \le \al$).
The upshot of \eqref{pole_order_condition} is that
\[
\mc{B}^{(i_0)}= \{x^{i_0}, x^{i_0}y, \dots, x^{i_0}y^{j(i_0)}\}. 
\]
In particular, we have $\mc{B}^{(i_0)}=\{x^{i_0}\}$ if and only if $i_0 \ge \al -\be$.

\medskip
Abusively, we will continue to use $W(\mc{B})$ (resp., $W(\mc{B}^{(i_0)})$) to denote its reduced version. Note that the submatrix of $W(\mc{B})$ spanned by the first $n(\al- (n-1) \be)$ rows and columns, which comprises all $W(\mc{B}^{(i_0)})$ with $1 \leq i_0 \leq \al- (n-1) \be -1$, contributes a (unit multiplier) factor of $(D_y^n x)^{n\binom{\al-(n-1)\be}{2}}$ to the lowest $y$-adically valued term of the Wronskian determinant $w(\mc{B})$. Indeed, it is easy to see that every diagonal entry of $W(\mc{B})$ in this range that belongs to $W(\mc{B}^{(i_0)})$ is $(D_y^n x)^{i_0}$, and that every entry above the diagonal in this range is zero modulo $D_y^{\eta} x$'s for $\eta=0,\dotsc,n-1$ (and $v_y(D_y^\eta x)>0$ for such $\eta$'s). Note that $W(\mc{B}^{(0)})$ itself contributes a trivial multiplicative factor of 1.

\medskip
Using \eqref{valuation_of_matrix_entry}, it is easy to identify the {\it tropical} $y$-adic image $W^{\rm trop}_{\rm \ast}(\mc{B})$ of the submatrix $W_{\rm \ast}(\mc{B})$ of $W(\mc{B})$ determined by the remaining rows and columns; removing columns in sets $\mc{B}^{i_0}$, each of cardinality $n$, for $i_0=0,\dotsc,\alpha-(n-1)\beta-1$, the number of remaining columns (same for rows) are:
\begin{align*}
    (\ell-g+1)-n(\alpha-(n-1)\beta)=\ell-g+1-\ell+2g=g+1
\end{align*}
since Riemann-Hurwitz formula for the superelliptic projection $\pi: X \rightarrow \mb{P}^1$ gives
\[
    2g-2=n(-2)+(n-1)(d+1)=-2n+(n-1)(n\beta+2)=n^2\beta-n\beta-2.
\]
So $W^{\rm trop}_{\rm \ast}(\mc{B})$ is a $(g+1) \times (g+1)$ matrix whose columns are stratified by the $y$-adic images $W^{\rm trop}_{\rm \ast}(\mc{B}^{(i_0)})$ of the corresponding reduced submatrices $W_{\rm \ast}(\mc{B}^{(i_0)})$ of $W(\mc{B}^{(i_0)})$, where $\al- (n-1) \be \leq i_0 \leq \al$. Indeed, the top row $V$ of $W^{\rm trop}_{\rm \ast}(\mc{B})$ is the concatenation $V= (V^{(i_0)})_{\al- (n-1) \be \leq i_0 \leq \al}$ of sequences
\[
V^{(i_0)}= ((i_0-\al+ (n-1) \be)n, \dots, (i_0-\al+ (n-1) \be)n+j(i_0)). 
\]
In any given column, entries of $W^{\rm trop}_{\rm \ast}(\mc{B})$ decrease by a unit for each successive row visited until they stabilize at zero.

\medskip
Note, moreover, that whenever $\det N(n,g,\ell)$ is nonzero, the {\it (tropical) permanent} of $W^{\rm trop}_{\rm \ast}(\mc{B})$ is precisely the local inflectionary multiplicity $\mu(\mc{B})$. It is also straightforward to write down the permanent explicitly. Indeed, it is precisely the sum of the diagonal entries of $W^{\rm trop}_{\rm \ast}(\mc{B})$, namely
{\tiny
\[
\begin{split}
\mu(\mc{B})&= \sum_{j=1}^{n-2} (n-j) \bigg[\bigg(\binom{j}{2}\be\bigg) + \bigg(\binom{j}{2}\be+j\bigg) + \bigg(\binom{j}{2}\be+2j\bigg) + \dotsb +\bigg(\binom{j}{2}\be+j(\be-1)\bigg)\bigg] \\
&\phantom{=\;\;}+\bigg(\binom{n-1}{2}\be\bigg)+ \bigg(\binom{n-1}{2}\be+(n-1)\bigg)+ \bigg(\binom{n-1}{2}\be+2(n-1)\bigg)+ \dotsb+ \bigg(\binom{n-1}{2}\be+(n-1)\be\bigg)\\
&=\sum_{j=1}^{n-1} (n-j)\bigg[\binom{j}{2}\be^2+ j \binom{\be}{2}\bigg] + \bigg(\binom{n-1}{2}\be+ (n-1) \be\bigg) \\
&=\frac{-3(n-1)^2n^2+2(n+1)(n-1)n(2n-1)-6(n-1)n^2-2(n-1)n(2n-1)+6(n-1)n^2}{24}\beta^2 \\
&\phantom{=\;\;} 
+ \frac{-3n^2(n-1)+(n-1)n(2n-1)+6(n-1)(n-2)+12(n-1)}{12}\beta \\
&=\frac{(n-1)n^2(n+1)}{24}\be^2+ \frac{(n-1)n(5-n)}{12}\be.
\end{split}
\]
}
Unlike in the hyperelliptic case, however, the partition $\la$ whose valuation \eqref{valuation_of_matrix_entry} realizes the minimum value recorded by the corresponding entry of $W^{\rm trop}_{\rm \ast}(\mc{B})$ is not unique in general when $n>2$, and as a result the local Wronskian determinant $w(\mc{B})$ does not single out a unique $y$-adically minimal monomial in the $y$-derivatives of $x$.

\medskip
To distinguish $y$-adically minimal monomials in the $y$-derivatives of $x$, we use $W^{\rm trop}_{\rm \ast}(\mc{B})$ as a blueprint. More precisely, as in \cite[Proof of Thm. 3.9]{CDH}, we replace $W^{\rm trop}_{\rm \ast}(\mc{B})$ by $W^{\rm trop}_{\rm \ast}(\mc{B})^{\pr}$ in which the top row remains the same, but whose entries in each column decrease one by one; $W^{\rm trop}_{\rm \ast}(\mc{B})$ and $W^{\rm trop}_{\rm \ast}(\mc{B})^{\pr}$ are analogous to $M$ and $M^{\pr}$ in {\it loc.cit.} respectively. 
To compensate for the nonuniqueness of minimally $y$-adically valued partitions, we are forced to make certain choices. More precisely, for every index $\al-(n-1)\be \leq i_0 \leq \al$ and for every index $k=0,\dots,j(i_0)$, we introduce a directed graph $PPG(i_0,n)$ whose set of vertices is the Pl\"ucker poset of a Grassmannian $G(i_0,n+i_0)$, and for which the vertices indexed by partitions $\lambda_1,\lambda_2$ 
are linked by a unique directed edge $\lambda_1 \rightarrow \lambda_2$ if and only if $\lambda_1 \le \lambda_2$ 
and $\text{wt}(\la_2)=\text{wt}(\la_1)+1$. We further define the {\it Pl\"ucker graph} $\mc{PG}(i_0,k)$ to be the full subgraph of $PPG(i_0,n)$ whose vertices are (indexed by) partitions of weight at least $n(\al-(n-1)\be)-k$ with $i_0$ parts; we let $\mc{P}(i_0,k)$ denote the set of maximal {\it paths} in $\mc{PG}(i_0,k)$; and we set $\mc{P}^*:= \Pi_{i_0,k} \mc{P}(i_0,k)$. For every vertex $\lambda \in \mc{PG}(i_0,k)$, there is an associated {\it occurrence weight} $ow(\lambda)$ equal to the number of paths in $\mc{P}(i_0,k)$ containing $\la$. 

\medskip
Using the combinatorial data from the Pl\"ucker graphs and posets introduced in the preceding paragraph, we now associate a matrix $W_*^p(\mc{B})$ to each $p \in \mc{P}^*$ as follows. Viewing $p$ as a tuple of paths in sets $\mc{P}(i_0,k)$ 
as above, we let $p(i_0,k)$ be the corresponding maximal path in $\mc{P}(i_0,k)$; and for each $\eta=0,\dotsc,g$ we let $p(i_0,k)(\eta)$ denote the partition in $p(i_0,k)$ of weight $n(\alpha-(n-1)\beta)-k+\eta$.\footnote{As a matter of convention, we decree $p(i_0,k)(\eta)$ to be the empty set $\emptyset$ whenever $\eta$ is at least the length of $p(i_0,k)$.} More generally, given a Pl\"ucker path $p^{\pr} \in \mc{P}(i_0,k)$, we define $p^{\pr}(\eta)$ in analogy to $p(i_0,k)(\eta)$. 
We define the column vector $W^{p^{\pr}}_*(\mc{B})$ so that each entry of $W^{p^{\pr}}_*(\mc{B})$ indexed by $\eta=0,\dotsc,g$ is either
\begin{equation}\label{eq:occurence_wt}
    \frac{1}{ow(p(i_0,k)(\eta))}\binom{i_0}{c_1,\dotsc,c_n}\prod_{m=1}^n (D^m_y x)^{c_m}
\end{equation}
whenever $p^{\pr}(\eta)=(1^{c_1},2^{c_2},\dotsc,n^{c_n})$, or else 0 when $p^{\pr}(\eta)=\emptyset$. Note that when $p^{\pr}(\eta) \neq \emptyset$, \eqref{eq:occurence_wt} exactly reproduces the corresponding monomial of the corresponding entry of $W_*(\mc{B})$ except for the renormalization factor $\frac{1}{ow(p(i_0,k)(\eta))}$.\footnote{Our hypothesis that $\text{char}(F)$ is either zero or sufficiently large ensures that our renormalization is well-defined.} 
The renormalization is specifically chosen to ensure that 
\[
    \sum_{p^{\pr} \in \mc{P}(i_0,k)} W^{p^{\pr}}_*(\mc{B}) \sim (i_0,k)^{\mathrm{th}} \text{ column of } W_*(\mc{B})
\]
in which $\sim$ means that the $y$-adic valuation vector of the difference of the two sides is larger (in every coordinate) than the corresponding value of $W^{\mathrm{trop}}_*(\mc{B})^{\pr}$. 

\medskip
We now define $W^p_*(\mc{B})$ to be the $(g+1) \times (g+1)$ matrix given by concatenating column vectors $W^{p(i_0,k)}_*(\mc{B})$ according to the lexicographic order on the set of pairs $(i_0,k)$. Similar to \cite[Proof of Thm. 3.9]{CDH}, the lowest $y$-adically valued terms of the two sides of the following equation are equivalent:
\begin{equation}\label{eq:decomp_W_*}
    \det W_*(\mc{B}) \sim \sum_{p \in \mc{P}^*} \det W^p_*(\mc{B}).
\end{equation}
Setting $M(p):=W^p_*(\mc{B})$, the proof of the second item follows.
\end{proof}

\begin{ex}\label{GV_example_n=2}
\emph{
    Let $n=2$, so $X$ is a hyperelliptic curve. In the notation of Theorem~\ref{Thm 3.9 generalized}, we have $\ell =2\al$ and $d=2\be +1$, $g=\be$, and $\al>\be$. 
    In this case every basis element $b \in \mc{B}$ is of the form $x^iy^j$ with $j \in \{0,1\}$, and $D^{v_y(b)}_yb|_{(0,0)}=(D^2_yx|_{(0,0)})^i$. Since $\mc{B}$ has elements $x^i$ for $0 \le i \le \al$ and $x^iy$ for $0 \le i < \al-\be$, it follows that $\Pi_i (D^{\mu_i}_yb_i)|_{(0,0)}$ is a power of $D^2_yx|_{(0,0)}$ with exponent
    \[
        \sum_{i=0}^{\al-\be-1}i+\sum_{i=0}^{\al}i= \frac{(\al-\be)(\al-\be-1)+(\al+1)\al}{2}=\frac{2\al(\al-\be)+\be(\be+1)}{2}=\al(\al-\be)+\binom{\be+1}{2}.
    \]
    Further, we have
    \[
        \mu(\mc{B})=\frac{(n-1)n^2(n+1)}{24}\be^2+ \frac{(n-1)n(5-n)}{12}\be=\binom{g+1}{2}.
    \]
    On the other hand, the Vandermonde matrix $N(2,\be,2\al)$ is of the form
    \[
        \begin{pmatrix}
            A & B \\
            0 & C
        \end{pmatrix}
    \]
    in which $A$ is an upper triangular matrix with all diagonal equal to one, and $C$ is a $(\be +1) \times (\be +1)$ matrix with $C_{i,j}=\binom{2(\al-\be)+2j}{2(\al-\be)+i}$ for all $0 \le i,j \le \be$. Therefore, whenever $\mathrm{char}(F)$ is either zero or sufficiently large, the lowest $y$-adic term of $w(\mc{B})$ is equal to
    \begin{align}\label{eq:n=2_Vandermonde}
    \begin{split}
        &(D^2_yx|_{(0,0)})^{\al(\al-\be)+\binom{\be+1}{2}} \cdot \det N(2,\be,2\al) \cdot y^{\binom{g+1}{2}} \\ &=(D^2_yx|_{(0,0)})^{\al(\al-\be)+\binom{\be+1}{2}} \cdot \det \left(\binom{2(\al-\be)+2j}{2(\al-\be)+i}\right)_{0 \le i,j \le \be} \cdot y^{\binom{g+1}{2}}
    \end{split}
    \end{align}
    which is in agreement with the first item of Theorem~\ref{Thm 3.9 generalized}.}
    
    \medskip
    \emph{
    It is also instructive to see how the second item of Theorem~\ref{Thm 3.9 generalized} translates in this particular case. 
    For every $i_0 \in [\al-\be=\al-(n-1)\be, \al]$, we have $j(i_0)=0$, so the corresponding columns of $W_*(\mc{B})$ are indexed by $(i_0,0)$. Moreover, for every such index $i_0$ and every $\eta =0,\dotsc,g=\be$, the pigeonhole principle implies that there is at most one partition of weight $n(\al-(n-1)\be)-k+\eta=2(\al-\be)+\eta$ with $i_0$ parts that fits into a $i_0 \times 2$ rectangle. (Indeed, whenever $i_0 \le 2(\al-\be)+\eta \le 2i_0$, it is $(1^{2(i_0-\al+\be)-\eta},2^{2(\al-\be)+\eta-i_0})$.) Therefore, the Pl\"ucker graph $\mc{PG}(i_0,0)$ is a single path given by such partitions, so $\mc{P}(i_0,0)$ is a singleton; and every occurrence weight is equal to one. As a result, $\mc{P}^*$ is also a singleton $\{p\}$, so we merely replace $W_*(\mc{B})$ by the matrix $W^p_*(\mc{B})$ defined by
    \[
        (W^p_*(\mc{B}))_{i_0,\eta}=\binom{i_0}{2(i_0-\al+\be)-\eta}(D^1_yx)^{2(i_0-\al+\be)-\eta}(D^2_yx)^{2(\al-\be)+\eta-i_0}
    \]
    for every pair of indices $\al-\be \le i_0 \le \al$ and $0 \le \eta \le \be$.
    It 
    is not hard to see that for any permutation of $(g+1)$ numbers, the corresponding term of $\det W^p_*(\mc{B})$ is equal to a scalar multiple of $(D^1_yx)^{\binom{g+1}{2}}(D^2_yx)^{(\al-\be)(\be+1)}$. The scalar coefficients of $W^p_*(\mc{B})$, in turn, comprise the Gessel-Viennot matrix $M(\al,\be)$ of \cite[Thm. 3.9, Rmk. 3.10]{CDH} with entries $M(\al,\be)_{w,v}=\binom{\al-\be+v}{2v-w}$ for $0 \le w,v \le \be$, where $v=i_0-\al+\be$ and $w=\eta$. 
    The upshot is that the lowest $y$-adically valued term of $w(\mc{B})$ is equal to that of
    \begin{equation}\label{eq:n=2_GV}
        (D^2_yx)^{2\binom{\al-\be}{2}}(\det M(\al,\be))(D^1_yx)^{\binom{g+1}{2}}(D^2_yx)^{(\al-\be)(\be+1)}=(\det M(\al,\be))(D^1_yx)^{\binom{g+1}{2}}(D^2_yx)^{\al(\al-\be)}
    \end{equation}
    which agrees with \cite[Thm. 3.9]{CDH}.\footnote{Note that $\ell$ (resp., $g$) in {\it loc.cit.} plays the role of $\al$ (resp., $\be$) here.}}
    
    \medskip
    \emph{
    To compare equations~\eqref{eq:n=2_Vandermonde} and \eqref{eq:n=2_GV}, we start by decomposing $x$ as a power series $x=cy^2+ (\text{higher-order terms in }y)$. The lowest $y$-adically valued terms of $D^1_yx$ and $D^2_yx$ are then $2cy$ and $c$ respectively. Applying linearity properties of the determinant, we obtain the following comparison identity for Vandermonde and Gessel--Viennot determinants:
    \begin{equation}\label{eq:n=2_comparison}
        \det N(2,\be,2\al)=\det \left(\binom{2(\al-\be)+2j}{2(\al-\be)+i}\right)_{0 \le i,j \le \be}=2^{\binom{g+1}{2}}\det M(\al,\be).
    \end{equation}
 }
\end{ex}

\begin{ex}\label{GV_example_n=3}
\emph{
When $n=3$, $d=4$, and $\ell=9$, we have $\al=3$, $\be=1$, and $g=3$ in the notation of Theorem~\ref{Thm 3.9 generalized}. The first item of Theorem~\ref{Thm 3.9 generalized} establishes that for every $b \in \mc{B}$, $b$ is of the form $x^iy^j$ with $j=0,1,2$; thus $D^{v_y(b)}_yb|_{0,0}=(D^3_yx|_{(0,0)})$. Much as in Example~\ref{GV_example_n=2}, we see that $\Pi_i(D^{\mu_i}_yb_i)|_{(0,0)}$ is a power of $D^3_yx|_{(0,0)}$ with exponent $\sum^{\al-2\be-1}_{i=0}i+\sum_{i=0}^{\al-\be-1}i+\sum_{i=0}^{\al}i=7$, while
\[
    \mu(\mc{B})=\frac{(n-1)n^2(n+1)}{24}\be^2+ \frac{(n-1)n(5-n)}{12}\be=4.
\]
Meanwhile, the Vandermonde matrix $N(3,3,9)$ is of the form
    $\begin{pmatrix}
        A & B \\
        0 & C
    \end{pmatrix}$,
in which $A$ is an upper triangular matrix with every diagonal entry equal to one, and
\begin{equation}\label{eq:n=3_Vandermonde_lower}
    C=
    \begin{pmatrix}
        \binom33 & \binom43 & \binom63 & \binom93\\
        0 & \binom44 & \binom64 & \binom94\\
        0 & 0 & \binom65 & \binom95\\
        0 & 0 & \binom66 & \binom96
    \end{pmatrix}.
\end{equation}
Therefore, whenever $\mathrm{char}(F)$ is greater than 7 or zero, the lowest $y$-adically-valued term of $w(\mc{B})$ is
\begin{equation}\label{eq:n=3_Vandermonde}
    (D^3_yx|_{(0,0)})^7 \cdot \det N(3,3,9) \cdot y^4= (D^3_yx|_{(0,0)})^7 \cdot \det C \cdot y^4=378 (D^3_yx|_{(0,0)})^7 y^4.
\end{equation}
}
\hspace{-5pt}\emph{On the other hand, 
the second item of Theorem~\ref{Thm 3.9 generalized} 
establishes that whenever $\text{char}(F) \neq 2,3$, the $y$-adically lowest-order term of $w(\mc{B})$ is 
equal to that of $\sum_{p \in \mc{P}^*} \det M(p)$. 
In this case, the columns of $W_*(\mc{B})$ are indexed by $(i_0,k)=(1,0),(1,1),(2,0),(3,0)$, and Figure~\ref{fig:Plucker_Graph} illustrates the corresponding Pl\"ucker graphs $\mc{PG}(i_0,k)$.} 
\begin{figure}
    \centering
    \includegraphics{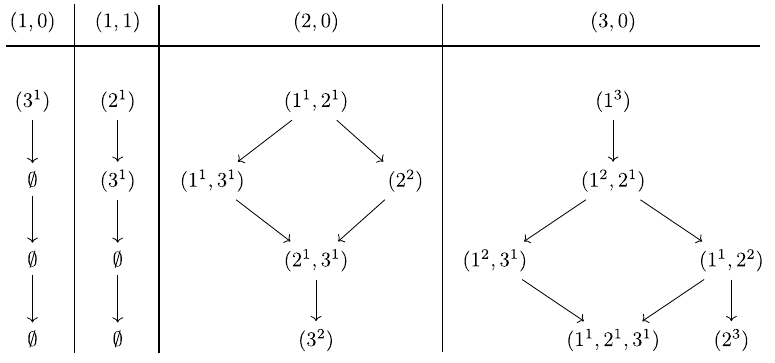}
    \caption{Pl\"ucker graphs $\mc{PG}(i_0,k)$ for $(i_0,k)=(1,0),(1,1),(2,0),(3,0)$} 
    \label{fig:Plucker_Graph}
\end{figure}
\emph{This, in turn, allows us to compute the set $\mc{P}^*$ of products of Pl\"ucker paths, along with the corresponding matrices $M(p)$ for every $p \in \mc{P}^*$. Summing their determinants, 
we deduce that the $y$-adically lowest-order term of $w(\mc{B})$ is equal to that of
\begin{equation}\label{eq:n=3_GV}
    9D^1_yx(D^2_yx)^2(D^3_yx)^4+2(D^2_yx)^4(D^3_yx)^3-3(D^1_yx)^2(D^3_yx)^.
\end{equation}
The lowest order terms of $D^1_yx$ and $D^2_yx$ are $3D^3_yx|_{(0,0)} \cdot y^2$ and $3D^3_yx|_{(0,0)} \cdot y$, respectively; 
it follows that equations~\eqref{eq:n=3_Vandermonde} and \eqref{eq:n=3_GV} are equivalent.
}
\end{ex}

\begin{rem}\label{rmk:non_GV}
\emph{
    Examples~\ref{GV_example_n=2} and \ref{GV_example_n=3} lead to interesting identities involving Vandermonde determinants; 
    see, e.g., equation~\eqref{eq:n=2_comparison}. 
    Indeed, every entry of $M(p):=W^p_*(\mc{B})$ is defined purely combinatorially by equation~\eqref{eq:occurence_wt}. Now let $\wt M(p)$ denote the matrix obtained from $M(p)$ by systematically replacing every monomial $\Pi_i (D^i_yx)^{c_i}$ in Hasse derivatives of $x$ by the corresponding monomial $\Pi_i t_i^{c_i}$ in formal variables $t_i$. The ``universal" matrix $\tilde M(p)$ depends exclusively on $n,\al,\be$ with $\frac{\al}{\be} >n-1$ (and not on the choice of the underlying superelliptic curve, once those parameters are fixed) and
    specializes to a matrix $\wt M(p)(\vec{t})$ of numbers under specializations of the formal vector $\vec{t}:=(t_0,t_1,\dotsc)$.
    The universal matrices $\wt M(p)$ are {\it generalized Gessel-Viennot matrices}, inasmuch as 
    when $n=2$, the specialization $\wt M(p)(1,1,1,\dotsc)$ recovers the Gessel-Viennot matrix of \cite[Thm 3.9 and Rmk 3.10]{CDH}.}

\medskip    
\emph{Note that according to the second item of Theorem~\ref{Thm 3.9 generalized}, the scalar coefficient of the lowest $y$-adic term of $w(\mc{B})$ may be rewritten as
    \begin{equation}\label{eq:lowest_y-adic_term}
        (D^n_yx)^{n\binom{\al-(n-1)\be}{2}}\sum_{p \in \mc{P}^*} \det \wt M(p)\left(\binom{n}{0}(D^n_yx)|_{(0,0)}y^n,\binom{n}{1}(D^n_yx)|_{(0,0)}y^{n-1},\dotsc,\binom{n}{n}(D^n_yx)|_{(0,0)}\right)
    \end{equation}
    since the lowest $y$-adic term of $D^i_yx$ is $\binom{n}{i}(D^n_yx)|_{(0,0)}y^{n-i}$ (and the formal variables $t_i$ in $\wt M(p)$ select for instances of the differential monomials $D^i_yx$). The argument used in the proof of Theorem~\ref{Thm 3.9 generalized} implies that \eqref{eq:lowest_y-adic_term} is equal to
    \begin{equation}\label{eq:lowest_y-adic_term_bis}
        (D^n_yx)^{n\binom{\al-(n-1)\be}{2}}\sum_{p \in \mc{P}^*} \det \wt M(p)\left(\binom{n}{0}(D^n_yx)|_{(0,0)},\binom{n}{1}(D^n_yx)|_{(0,0)},\dotsc,\binom{n}{n}(D^n_yx)|_{(0,0)}\right) \cdot y^{\mu(\mc{B})}
    \end{equation}
    Comparing \eqref{eq:lowest_y-adic_term_bis} against the first item of Theorem~\ref{Thm 3.9 generalized} and substituting ones for instances of $D^n_yx|_{(0,0)}$, we now obtain
    \begin{equation}\label{eq:Vandermonde_vs_GV}
        \det N(n,g,\ell)=\sum_{p \in \mc{P}^*}\det \wt M(p)\left(\binom{n}{0},\binom{n}{1},\dotsc,\binom{n}{n}\right)
    \end{equation}
    which in turn generalizes the Vandermonde determinant identities of Examples~\ref{GV_example_n=2} and \ref{GV_example_n=3}.}
\end{rem}

\subsection{Hasse inflection polynomials}\label{sec:Hasse_inflection_polynomials}

As before, assume $X$ is a superelliptic curve affinely presented by $y^n=f(x)$. Hereafter, also assume that $\mathrm{char}(F)$ does not divide $n$. Given positive integers $\ell$ and $m$, we define the {\it $(\ell,m)$-th atomic Hasse inflection polynomial} $P_m^{\ell}(x)$ according to
\begin{equation}\label{atomic_inflection_polys}
D^m y^{\ell}= f^{-m} y^{\ell} \cdot P_m^{\ell}(x)
\end{equation}
where $D=D_x$ denotes Hasse differentiation with respect to $x$. Here we view equation \eqref{atomic_inflection_polys} as an equality of rational functions on $X$.
The characteristic property of $P^{\ell}_m$ is that its zeroes parameterize the $x$-coordinates of zeroes of $D^m y^{\ell}$, or equivalently the $\ov{F}$-rational inflection points of the linear series on $X$ with basis $\{1,x,\dots,x^{m-1}; y^{\ell}\}$, supported away from the superelliptic ramification locus $R_{\pi}$.

\begin{prop}\emph{(Generalization of \cite[Prop. 3.17]{CDH})}\label{prop:infl_recursion}
For each fixed value of positive integer $\ell=1,\dots,n-1$, the atomic Hasse inflection polynomials $P_m^{\ell}(x)$ are specified recursively by
\[
P_{m+1}^{\ell}= \frac{1}{m+1}(D^1 P_m^{\ell} \cdot f+ P_m^{\ell} \cdot D^1 f \cdot (-m+u))
\]
where $u=\frac{\ell}{n}$ and $m \geq 1$, subject to the seed datum $P_1^{\ell}= u \cdot D^1f$.
\end{prop}

\begin{proof}
Differentiating the affine presentation $y^n= f(x)$ for $X$ yields $D^1 y= \frac{1}{n} f^{-1}y D^1 f$ and consequently 
\[
D^1 y^{\ell}= \ell y^{\ell-1} \cdot D^1y 
= u \cdot D^1 f \cdot f^{-1}y^{\ell}
\]
which justifies our definition of $P_1^{\ell}$. Note that whenever $\mathrm{char}(F) \neq 0$, the fact that we may meaningfully ``divide" by $n$ follows from the same ``spreading out" argument used in the proof of \cite[Prop. 3.17]{CDH}. On the other hand, differentiating the defining equation~\eqref{atomic_inflection_polys} for Hasse inflection polynomials yields 
\[
\begin{split}
    D^1 D^m y^{\ell}&= (D^1 P^{\ell}_m) f^{-m} y^{\ell}+ P^{\ell}_m \cdot (-m f^{-(m+1)}D^1f \cdot y^{\ell}+ f^{-m} \cdot \ell y^{\ell-1} \cdot D^1 y) \\
    &= (D^1 P^{\ell}_m) f^{-m} y^{\ell}+ P^{\ell}_m \cdot \bigg(-m f^{-(m+1)}D^1f \cdot y^{\ell}+ f^{-m} \cdot \ell y^{\ell-1} \cdot \frac{1}{n} f^{-1}y D^1 f \bigg) \\
    &= f^{-(m+1)} y^{\ell} (D^1 P_m^{\ell} \cdot f+ P_m^{\ell} \cdot D^1 f \cdot (-m+u)).
\end{split}
\]
The desired recursion now follows from the fact that $D^1 D^m= (m+1) D^{m+1}$.
\end{proof}

\subsection{Inflectionary varieties from superelliptic families}
Given a flat family of superelliptic curves $X_{(\la_i)}: y^n=f_{(\la_i)}(x)$ in a finite number of parameters $\{\la_i\}$, we refer to the hypersurface in the affine space with coordinates $x$ and $(\la_i)$ cut out by the atomic inflection polynomial $P^{\ell}_m$ of the preceding subsection as the {\it $(\ell,m)$-th atomic inflectionary variety} associated to the family $X_{(\la_i)}$.

\begin{rem}\label{rmk:infl_recursion_ext}
\emph{
The same argument deployed in the proof of \cite[Prop. 3.17]{CDH} shows that Proposition~\ref{prop:infl_recursion} may be extended to families of superelliptic curves; but the most general statement along these lines requires replacing the coefficients of $f(x)$ by sections of certain line bundles (for example, see \cite{F} when $\mathrm{char}(F)=0$ and $n=2$). For families parameterized by {\it rings}, however, it is easy to be more explicit. Namely, whenever $X: y^n= f(x)$ is a superelliptic curve defined over a ring $R$, the corresponding Hasse inflection polynomials are elements of $R[\frac1n][x]$. For example, whenever $X: y^n= f(x)$ is defined over $\mb{Z}$, its Hasse inflection polynomials are all defined over $\Z[\frac1n]$. This is optimal, as $\Z[\frac1n]$ is a natural ``ring of definition" for $X$ itself as a separable degree $n$ cover of $\mathbb{P}^1$.}
\end{rem}

\subsection{A determinantal formula}
Inflection points of the complete series $|\mc{O}(\ell \infty_X)|$ on $X$ supported on the complement $R_{\pi}^{\complement}$ of the superelliptic ramification locus are computed by local Wronskians of partial derivatives with respect to $x$ of the monomial basis $\mc{B}$ of Lemma~\ref{monomial_basis}. Just as in \cite[Lem. 2.1]{CG2}, these local Wronskians are naturally related to explicit determinants in the atomic Hasse inflection polynomials introduced above. In order to make this precise, we will keep the same numerological hypotheses as in Theorem~\ref{Thm 3.9 generalized}. Applying equation \eqref{pole_order_condition} in the proof of that result, we see that for every fixed choice of nonzero $y$-exponent $j_0$, there are precisely $\al- \be j_0$ monomials $x^iy^{j_0}$ in $\mc{B}$, which comprise a distinguished subset $\mc{B}_{(j_0)}$. We now order the elements of $\mc{B}$ according to increasing $y$-exponent, starting with the powers of $x$ that belong to $\mc{B}_{(0)}$; and within each block $\mc{B}_{(j_0)}$, we order elements according to increasing $x$-exponent. With respect to this ordering, the (partial $x$-derivatives of the) elements of $\mc{B}_{(0)}$ contribute an identity submatrix $I$ to the local Wronskian $W(\mc{B})$, and correspondingly the local Wronskian {\it determinant} is equal to that of the complement $W_{\ast}(\mc{B})$ of $I$. Moreover, column-reducing as in the proof of Theorem~\ref{Thm 3.9 generalized}, we may systematically replace every entry of $W_{\ast}(\mc{B})$ of the form $D^k(x^i y^{j_0})$ by $D^{k-i}(y^{j_0})$, or equivalently, by $f^{-(k-i)} y^{j_0} \cdot P^{j_0}_{k-i}(x)$. The determinant of the resulting matrix is equal to, up to an irrelevant nonzero rational function of $f$ and $y$, the determinant of the matrix $\wt{W}_{\ast}(\mc{B})$  obtained from $W_{\ast}(\mc{B})$ by systematically replacing every $D^k(x^i y^{j_0})$ by $P^{j_0}_{k-i}(x)$.

\begin{thm}\label{inflection_poly_det}\emph{(Generalization of \cite[Lem. 2.1]{CG2})} Assume that $\ell \geq 2g+n-1$, $\ell=n\al$ and $d= n\be+1$, where $\al$ and $\be$ are positive integers for which $\frac{\al}{\be}>n-1$. There exists a homogeneous polynomial $Q_{\al,\be} \in \mb{Z}[t_{i,j}: 1 \leq j \leq n-1, \be j+1 \leq i \leq \ell-g]$ of degree $\ell-g-\al=\frac{(n-1)(2\al-n\be)}{2}$ for which the zeroes of $Q_{\al,\be}|_{t_{i,j}=P^j_i(x)}$ are the $x$-coordinates of the $\ov{F}$-inflection points of $|\mc{O}(\ell \infty_X)|$ supported along $R_{\pi}^{\complement}$. Explicitly, $Q_{\al,\be}|_{t_{i,j}=P^j_i(x)}$ is the determinant of the matrix $\wt{W}_{\ast}(\mc{B})$ described above.
\end{thm}

\begin{proof}
The proof follows easily from the discussion above; the salient points here are that 1) the degree of $Q_{\al,\be}$ is equal to the width of $W_{\ast}(\mc{B})$, and 2) equation~\eqref{pole_order_condition} yields $i \geq \al+1-(\al-\be j-1)=\be j+2$ for every index $j=1,\dots,n-1$.
\end{proof}

\begin{ex}
When $n=3$, $d=4$, and $\ell=9$, Theorem~\ref{inflection_poly_det} establishes that the $x$-coordinates of those $\ov{F}$-inflection points of $|\mc{O}(9 \infty_X)|$ supported along $R_{\pi}^{\complement}$ comprise the zeroes of the determinant of
\[
\begin{pmatrix}
P^1_4 & P^1_3 & P^2_4 \\
P^1_5 & P^1_4 & P^2_5 \\
P^1_6 & P^1_5 & P^2_6
\end{pmatrix}.
\]
\end{ex}

\section{Inflectionary curves from superelliptic Legendre and Weierstrass pencils}\label{sec:legendre_and_weierstrass_pencils}
In \cite{CDH,CG1,CG2}, we studied $F$-rationality phenomena for inflectionary {\it curves} $\mc{C}_m$ defined by atomic inflection polynomials $P_m$ built out of one-parameter Legendre and Weierstrass pencils of elliptic curves, with a focus on those cases in which $F=\mb{R}$ or $F=\mb{F}_p$ for an odd prime $p$. In this setting, $\mc{C}_m$ is naturally a singular plane curve defined over $\mb{Z}$, or else its reduction modulo $p$. Moreover, the birational geometry of inflectionary curves $\mc{C}_m$ varies depending upon whether the underlying pencil of elliptic curves is of Legendre or Weierstrass type. In particular, the inflectionary curves $\mc{C}_m, 2 \leq m \leq 5$ derived from the Legendre pencil have rational desingularizations, whereas the Weierstrass inflectionary curve $\mc{C}_2$ is {\it elliptic}, with complex multiplication over $\mb{Q}(\sqrt{-3})$; see \cite[Prop. 4.2]{CDH}. 
The following table summarizes our conjectures to date regarding the salient features of atomic inflectionary curves $\mc{C}_m$, $m \geq 2$ associated to Legendre and Weierstrass pencils over 
a field $F$ whose characteristic is either zero or sufficiently positive. {\it Singularities} refer to those of the base extension of $\mc{C}_m$ to $\ov{F}$.

\begin{center}
    \begin{tabular}{|p{3cm}|p{3cm}|p{3cm}|p{3cm}|p{3cm}|}
    \hline
        Elliptic pencil type & Geometrically irreducible? & Number of singularities & Singularity types & Geometric genus $p_g$\\
     \hline
    Legendre & yes, unless $m=3$; $\mc{C}_3$ is the union of 3 conics & 3 for every $m \geq 2$ & Each is the transverse union of $(m-2)$ smooth branches and a cusp of type $y^2=x^{n+1}$ & $p_g=~\max(0,\binom{2m-1}{2}-3 \lfloor \frac{(m-1)^2}{2} \rfloor -3m+3)$
     \\
     \hline
     Weierstrass & yes & 1 if $m=2$; {\color{blue}3 for every $m \geq 3$, when $\mc{C}_m$ is compactified inside of $\mb{P}(1,2,1)$} & {\color{blue}See Conjecture~\ref{inner_edge_conjecture} and accompanying discussion} &$p_g(\mc{C}_2)=1$; {\color{blue}$\lceil \frac{(m-1)^2}{4} \rceil$ if $m \geq 3$} \\
     \hline
    \end{tabular}
\end{center}

\medskip
In this section, we further develop this conjectural picture to include atomic inflectionary curves associated to {\it superelliptic} Legendre and Weierstrass pencils with affine presentations $y^n=x^a(x-1)^b(x-\la)^c$ and $y^n=x^3+ \la x+ 2$, respectively. The geometry of superelliptic Legendre pencils depends on the value of the quotient $u=\frac{\ell}{n}$, and is closely linked to algebraic differential equations and hypergeometric series; see, e.g., \cite{HT,OH}. The conjectural number of singularities (3) of Weierstrass inflectionary curves appears in blue as it is not stated explicitly in our earlier papers \cite{BCG,CG1,CG2,CDH}. However iterating the characteristic recursion for atomic inflection polynomials leads to the expectation (formalized in Conjecture~\ref{conj:weierstrass_sing} below) that whenever $u=\frac{1}{2}$, the corresponding inflectionary curves $\mc{C}^{\ell}_m$ with $m \geq 3$ are always singular exactly in the points $q_j=[\zeta^{-j}: -3\zeta^j:1]$, $j=0,1,2$ in the weighted projective plane $\mb{P}(1,2,1)$\footnote{Here the weights are those of the coordinates $x$, $\la$, and $z$, respectively.}, where $\zeta$ is a cube root of unity. We will have more to say about this later, including an explicit characterization of singularity types (see Conjecture~\ref{inner_edge_conjecture}) and geometric genera (see Conjecture~\ref{conj:weierstrass_geom_genus}). We pay special attention to the $u=\frac{1}{2}$ case, as it applies to all hyperelliptic pencils; indeed, the notation $\mc{C}_m$ used in \cite{BCG,CG1,CG2} corresponds to what we call $\mc{C}^{\ell}_m$ here whenever $u=\frac{1}{2}$.

\subsection{Symmetries and singularities of superelliptic Legendre inflectionary curves}\label{sec:symmetries_of_Legendre_curves}
We begin by proving a generalization of \cite[Lem. 4.1]{CG1}, which describes the symmetries of certain Legendre inflectionary curves. 

\begin{thm}\label{superell_inflection_poly_symmetry}
Given positive integers $\ell, m, n$ with $n \geq 2$ and $a \in \mb{N}_{>0}$, the atomic inflection polynomial $P^{\ell}_m=P^{\ell}_m(x,\la)$ derived from the superelliptic Legendre pencil $y^n=x^a(x-1)^a(x-\la)^a$ has symmetries
\begin{equation}\label{legendre_inflection_poly_symmmetry}
    P^{\ell}_m(x,\la)= P^{\ell}_m(x,z) \text{ and } P^{\ell}_m(x+1,\la+1)=(-1)^{am}P^{\ell}_m(-x,-\la).
\end{equation}
Here by $P^{\ell}_m(x,z)$ we mean the polynomial obtained from $P^{\ell}_m(x,\la)$ by first homogenizing with respect to $z$, and then dehomogenizing with respect to $\la$.
\end{thm}

\begin{proof}
The proof of \cite[Lem. 4.1]{CG1} in fact carries over verbatim, but for completeness we give the argument. Accordingly, let $f(x,\la):=x^a(x-1)^a(x-\la)^a$; note that $f(x,\la)$ becomes $f(x,z)$ when $\la$ is replaced by $z$. The first symmetry now holds by induction using Proposition~\ref{prop:infl_recursion}, as it is preserved by differentiation with respect to $x$. Similarly, the second symmetry follows from induction on $m$ using Proposition~\ref{prop:infl_recursion}, together with the facts that 1) $D_x f$ and consequently $(-m+ u) \cdot D_x f$, has the second symmetry (with respect to $m=1$); and 2) 
$D_x P_m^{\ell} \cdot f$ also has the second symmetry (with respect to $m+1$ instead). \end{proof}

One immediate consequence of Theorem~\ref{superell_inflection_poly_symmetry} is that when $a=b=c$ and our base field is $F=\mb{Q}$, the projective closure $\mc{C}^{\ell}_m \sub \mb{P}^2_{x,\la,z}$ of the curve defined by $P^{\ell}_m$ is singular in 
$p_1=[0:0:1]$, $p_2=[0:1:0]$, and $p_3=[1:1:1]$, and that all three singularities are isomorphic over $\mb{Q}$. Proposition~\ref{prop:infl_recursion} together with induction also shows that the inflectionary curve $\mc{C}^{\ell}_m$ derived from the superelliptic pencil $y^n=x^a(x-1)^b(x-\la)^c$ is always {\it singular} in $p_1$, $p_2$, and $p_3$; but the corresponding singularity types are in general distinct.

\begin{conj}\label{singularity_support_conj}\emph{(Generalization of \cite[Conj. 4.3]{CG1})}
Suppose that $\text{char}(F)$ is either zero or sufficiently positive. For all $\ell$ and $m$, the inflectionary curve $\mc{C}^{\ell}_m$ derived from the pencil $y^n=x^a(x-1)^b(x-\la)^c$ is nonsingular away from $p_1$, $p_2$, and $p_3$.
\end{conj}

The Newton polygons of the atomic inflection polynomials $P^{\ell}_m$ are also significant, insofar as they yield critical information about the arithmetic genus and singularities (and correspondingly, the geometric genus) of $\mc{C}^{\ell}_m$.
Proposition~\ref{prop:infl_recursion} leads naturally to the following result, which gives a prediction for the Newton polygon of  inflection polynomials under suitable genericity hypotheses.

\begin{thm}\label{generic_legendre_newton_polygon}
Given positive integers $a$, $b$, $c$, $\ell$ and $n$, suppose that for every positive integer $m$, the atomic inflection polynomial $P^{\ell}_m$ derived from the superelliptic Legendre pencil $y^n=x^a(x-1)^b(x-\la)^c$ has generic support. Then for every $m$, the associated Newton polygon is
\[
{\rm New}(P^{\ell}_m)= \text{Conv}((ma+mc-m,0), (ma+mb+mc-m,0), (ma-m,mc), (ma+mb-m,mc)).
\]
\end{thm}

\begin{proof}
Letting $f:=x^a (x-1)^b (x-\la)^c$ as before, and letting $\oplus_{\rm M}$ denote the Minkowski sum of polygons, the Newton polygon of $f$ is given explicitly by
\begin{equation}\label{newton_poly_of_f}
\begin{split}
{\rm New}(f)&=(a,0) \oplus_{\rm M} \text{Conv}((0,0),(b,0)) \oplus_{\rm M} \text{Conv}((0,c), (c,0)) \\
&=\text{Conv}((a+c,0),(a+b+c,0),(a,c),(a+b,c)).
\end{split}
\end{equation}
It follows from \eqref{newton_poly_of_f} that 
\[
{\rm New}(P_1^{\ell})={\rm New}(D^1_x f)= \text{Conv}((a+c-1,0),(a+b+c-1,0),(a-1,c),(a+b-1,c)).
\]
In particular, Theorem~\ref{generic_legendre_newton_polygon} holds whenever $m=1$. Now suppose that $m>1$, and that Theorem~\ref{generic_legendre_newton_polygon} holds for ${\rm New}(P^{\ell}_{m-1})$. We then have
\[
\begin{split}
{\rm New}(P^{\ell}_{m-1})&=\text{Conv}(((m-1)a+(m-1)c-m+1,0), ((m-1)a+(m-1)b+(m-1)c-m+1,0), \\
&((m-1)a-m+1,(m-1)c),((m-1)a+(m-1)b-m+1,(m-1)c)).
\end{split}
\]
Genericity of support now implies that
{\small
\[
\begin{split}
{\rm New}(P^{\ell}_m)&= \text{Conv}(\text{Conv}((m-1)a+(m-1)c-m,0), ((m-1)a+(m-1)b+(m-1)c-m,0),((m-1)a-m,(m-1)c),\\
&((m-1)a+(m-1)b-m,(m-1)c))
\oplus_M \text{Conv}((a+c,0),(a+b+c,0),(a,c),(a+b,c))\\
& \bigcup \text{Conv}(((m-1)a+(m-1)c-m+1,0), ((m-1)a+(m-1)b+(m-1)c-m+1,0),((m-1)a-m+1,(m-1)c),\\
&((m-1)a+(m-1)b-m+1,(m-1)c))\oplus_M \text{Conv}((a+c-1,0),(a+b+c-1,0),(a-1,c),(a+b-1,c)))
\end{split}
\]
}
and the desired result follows.
\end{proof}

\begin{rem}
\emph{Whenever $\min(a,b,c)>1$, the superelliptic curve $X: y^n=x^a(x-1)^b(x-\la)^c$ is singular; however, $X$ is birational to a smooth curve $\wt{X}$ obtained via blow-ups along the superelliptic ramification locus. As a result, the local coordinates $x$ and $y$ unambiguously specify local coordinates on $\wt{X}$ along the preimage $U$ of $R_{\pi}^{\complement}$, and the inflection polynomials $P^{\ell}_m(x)$ compute the inflection of linear series with bases $\{1,x,\dots,x^{m-1}; y^{\ell}\}$ on $\wt{X}$ along the open locus $U$.}
\end{rem}

The inflection polynomial $P^{\ell}_m$ derived from a given superelliptic family may fail to have generic support. Indeed, Proposition~\ref{prop:infl_recursion} implies that generically the coefficient of each monomial in $x$ and $\la$ in the expansion of $P^{\ell}_m$ is a polynomial of degree $m$ in $u(\ell,n)=\frac{\ell}{n}$, which may vanish for special values of $u$. Indeed, in practice it will often be the case that the coefficients of those monomials (corresponding to lattice points) that lie along the outer edges of $\text{New}(P^{\ell}_m)$ will split $F$-linearly in $u$; and the (roots of the) linear factors single out special values of $u$ where the behavior of $\text{New}(P^{\ell}_m)$ deviates from the generic behavior predicted by Minkowski sums. In writing down these coefficients explicitly, we will make frequent use of the following combinatorial devices.

\begin{dfn}\label{def:factorials}
Given $k \in \mb{N}$, $(w)_k:=w (w-1) \cdots (w-k+1)$ (resp., $(w)^k:= w (w-1) \cdots (w-k+1)$) denotes the $k$-th \emph{falling (resp., rising) factorial} of $w$. Similarly, $(\!(w)\!)_k:= w(w-2) \cdots (w-2k+2)$ (resp., $(\!(w)\!)_k:= w(w-2) \cdots (w-2k+2)$) denotes the $k$-th \emph{double falling (resp., rising) factorial} of $w$.
\end{dfn}

Our next result establishes that the Newton polygon $\text{New}(P^{\ell}_m)$ is generic whenever the underlying superelliptic family is of Legendre type and $u$ is sufficiently small, i.e., when $n$ is large relative to $\ell$.

\begin{thm}\label{generic_support_legendre_newton_polygon}
Suppose that $\text{char}(F)$ is either zero or sufficiently positive. Given positive integers $a$, $b$, $c$, $\ell$ and $n$ as above, the Newton polygon of the inflection polynomial $P^{\ell}_m$ derived from the superelliptic Legendre family $y^n=x^a(x-1)^b(x-\la)^c$ is
\[
{\rm New}(P^{\ell}_m)= \text{Conv}((ma+mc-m,0), (ma+mb+mc-m,0), (ma-m,mc), (ma+mb-m,mc))
\]
whenever $n>(a+b+c)\ell$.
\end{thm}

\begin{proof}
We will prove a stronger statement by induction: that
the coefficients in $P^{\ell}_m$ of the critical monomials $x^{ma+mc-m}$, $x^{ma+mb+mc-m}$, $x^{ma-m}\la^{mc}$ and $x^{ma+mb-m}\la^{mc}$ are $\frac{(-1)^{bm}}{m!}((a+c)u)_m$, 
$\frac{1}{m!}((a+b+c)u)_m$, $\frac{(-1)^{(b+c)m}}{m!}(au)_m$, and $\frac{(-1)^{cm}}{m!}((a+b)u)_m$, respectively. This will imply, in particular, that each of these critical monomials is nonvanishing whenever $n>(a+b+c)\ell$.

\medskip
\noindent For notational convenience, we let
\[
v^1_m=(ma+mc-m,0), v^2_m= (ma+mb+mc-m,0), v^3_m= (ma-m,mc) \text{ and } v^4_m= (ma+mb-m,mc)
\]
and further let $v^i_{m,-}:= v^i_m- (1,0)$ and $v^i_{m,+}:=  v^i_m+ (1,0)$
for every positive integer $m$. We will use $[v^i_m]P$ as a  shorthand for the coefficient of the term in the expansion of $P=P(x,\la)$ associated with the monomial indexed by $v^i_m$. Proposition ~\eqref{prop:infl_recursion} now implies that
\begin{equation}\label{inductive_coefficient_relation}
[v^i_{m+1}]P^{\ell}_{m+1}= \frac{1}{m+1} ([v^i_{m,-}]D^1 P^{\ell}_m \cdot [v^i_{1,+}]f+ [v^i_m] P^{\ell}_m \cdot [v^i_1] D^1 f \cdot (u-m))
\end{equation}
for every $i=1,2,3,4$. It now suffices to argue inductively case by case for each value of $i$ using \eqref{inductive_coefficient_relation}.

\medskip
In the interest of space (and because the other cases are analogous), we give the argument when $i=1$ and leave the remaining cases to the reader. We have $[v^1_{1,+}]f=(-1)^b$ and $[v^1_1]D^1 f=(-1)^b(a+c)$; as $P^{\ell}_1=u D^1 f$, it follows that $[v^1_1]P^{\ell}_1=(-1)^b(a+c)u$, and the claim in this case holds when $m=1$. Now assume the claim holds for $m$; we then have $[v^1_m]P^{\ell}_m=\frac{(-1)^{mb}}{m!}((a+c)u)_m$ and $[v^1_{m,-}]D^1 P^{\ell}_m= m(a+c-1) \cdot \frac{(-1)^{mb}}{m!}((a+c)u)_m$, and applying \eqref{inductive_coefficient_relation} we deduce that
\[
\begin{split}
 [v^1_{m+1}]P^{\ell}_{m+1}&= \frac{1}{m+1}\bigg(m(a+c-1) \cdot \frac{(-1)^{(m+1)b}}{m!}((a+c)u)_m+\frac{(-1)^{(m+1)b}}{m!}((a+c)u)_m \cdot (a+c) \cdot (u-m)\bigg)\\
 &= \frac{(-1)^{(m+1)b}}{(m+1)!}((a+c)u)_m \cdot (m(a+c-1)+ (a+c)(u-m)) \\
 &= \frac{(-1)^{(m+1)b}}{(m+1)!}((a+c)u)_m \cdot ((a+c)u-m) 
\end{split}
\]
as desired.
\end{proof}

\begin{rem}
\emph{We suspect that a stronger version of Theorem~\ref{generic_support_legendre_newton_polygon} holds: namely, that whenever $n>(a+b+c)\ell$, the {\it support} of $P^{\ell}_m$ is itself generic. This would be substantially more difficult to prove, as the coefficients of monomials $x^i\la^j$ corresponding to interior points of $\text{New}(P^{\ell}_m)$ do not split into $u$-linear factors over $\mb{Q}$ in general; moreover, there is no obvious analogue of the inductive coefficient relation \eqref{inductive_coefficient_relation}, which depends upon the $v^i_m$ lying along the boundary of $\text{New}(P^{\ell}_m)$.}
\end{rem}

We next compute $\text{New}(P^{\ell}_m)$ whenever $n=2\ell$ and $(a,b,c)=(1,1,1)$. When $\ell=1$, this case is the focus of \cite[Conj. 2.4]{CG2}.

\begin{thm}\label{support_legendre_newton_polygon_z=2}
Suppose that $\text{char}(F)$ is either zero or sufficiently positive. For every positive integer $m \geq 2$, the Newton polygon of the inflection polynomial $P^{\ell}_m$ derived from $y^n=x(x-1)(x-\la)$ is
\[
\text{New}(P^{\ell}_m):= \text{Conv}((0,m),(m-2,m),(m-2,2),(2m-1,1),(2m-1,0),(2m,0))
\]
whenever $n=2\ell$. 
\end{thm}

\begin{proof}
Much as in the proof of Theorem~\ref{generic_support_legendre_newton_polygon}, 
we will explicitly identify the coefficients of those monomials $x^i\la^j$ in the expansion of $P^{\ell}_m$ corresponding to the vertices of the putative Newton polygon; however, we will also need to prove additional vanishing statements for coefficients that arise because $u=\frac{1}{2}$. According to Theorem~\ref{generic_support_legendre_newton_polygon}, we have $[\la^m]P^{\ell}_m=\frac{1}{m!}(u)_m$ and $[x^{2m}]P^{\ell}_m= \frac{1}{m!}(3u)_m$ for every $m$. For every integer $m \geq 2$, let $v^1_m= (m-2,m)$, $v^2_m=(m-2,2)$, $v^3_m=(2m-1,0)$, and $v^4_m=(2m-1,1)$. Using $u=\frac12$, we claim that moreover
\begin{equation}\label{explicit_coefficients_legendre_u=1/2}
\begin{split}
[v^1_m]P^{\ell}_m=[v^2_m]P^{\ell}_m&= -\frac{1}{8} \;\; \text{if } m \geq 2 \\
[v^3_m]P^{\ell}_m=[v^4_m]P^{\ell}_m&= 
  -\frac{2}{(m-1)!} u(3u-1)_{m-1} \;\; \text{if }m \geq 2
\end{split}
\end{equation}
and that $[x^i \la^j]P^{\ell}_m=0$ for all $(i,j) \notin \text{Conv}((0,m),v^1_m,v^2_m,v^3_m,v^4_m,(2m,0))$. Now let $v^5_m=(2m-2,1)$, and define $L^{\ell}_m$ to be the union of two rays $L^{\ell,1}_m$ and $L^{\ell,2}_m$ emanating from $v^5_m$ with slopes $-1$ and 0 respectively. We then claim that furthermore
\begin{equation}\label{explicit_coefficient_v^5_m}
    [v^5_m]P^{\ell}_m=\frac{1}{(m-1)!}((4m-1)u-m)u \cdot (3u-2)_{m-2} \;\; \text{if } m \geq 2
\end{equation}
and that $[x^i\la^j]P^{\ell}_m=0$ for every $(i,j) \in L^{\ell}_m \setminus \{v^5_m\}$ (here we use the fact that $u=\frac12$).

\medskip
Indeed, the required conditions clearly hold when $m \in \{2,3,4\}$. 
Arguing inductively, assume that $m \geq 3$; that $[x^i \la^j]P^{\ell}_m=0$ for all $(i,j) \notin \text{Conv}((0,m),v^1_m,v^2_m,v^3_m,v^4_m,(2m,0))$ and $(i,j) \in L^{\ell}_m \setminus \{v^5_m\}$; and that the explicit coefficient formulas \eqref{explicit_coefficients_legendre_u=1/2} and \eqref{explicit_coefficient_v^5_m} are operative. It follows, in particular, that $NP^{\ell}_m=\text{New}(P^{\ell}_m)$.
Proposition~\ref{prop:infl_recursion} now implies that $\text{New}(P^{\ell}_{m+1})$ lies inside
\[
\begin{split}
NP^{\ell,\text{out}}_{m+1}:=&\text{Conv}(\text{New}(D^1 P^{\ell}_m) \oplus_M \text{New}(f) \bigcup \text{New}(P^{\ell}_m) \oplus_M \text{New}(D^1 f)) \\
=&\text{Conv}((0,m+1),(m-1,m+1),(m-1,2), (2m,2), (2m-1,0), (2m+2,0)).
\end{split}
\]
See Figure~\ref{Newton_Pm_BPm} for a comparison of $NP^{\ell}_{m+1}$ and $NP^{\ell,\text{out}}_{m+1}$ when $m=3$. To prove $\text{New}(P^{\ell}_{m+1})=NP^{\ell}_{m+1}$, it suffices to show that $\text{New}(P^{\ell}_{m+1}) \supset NP^{\ell}_{m+1}$ and $[x^i\la^j]P^{\ell}_{m+1}=0$ for all $(i,j) \in NP^{\ell,\text{out}}_{m+1} \setminus NP^{\ell}_{m+1}$.

\begin{figure}[h]       
    \fbox{\includegraphics[width=0.45\linewidth]{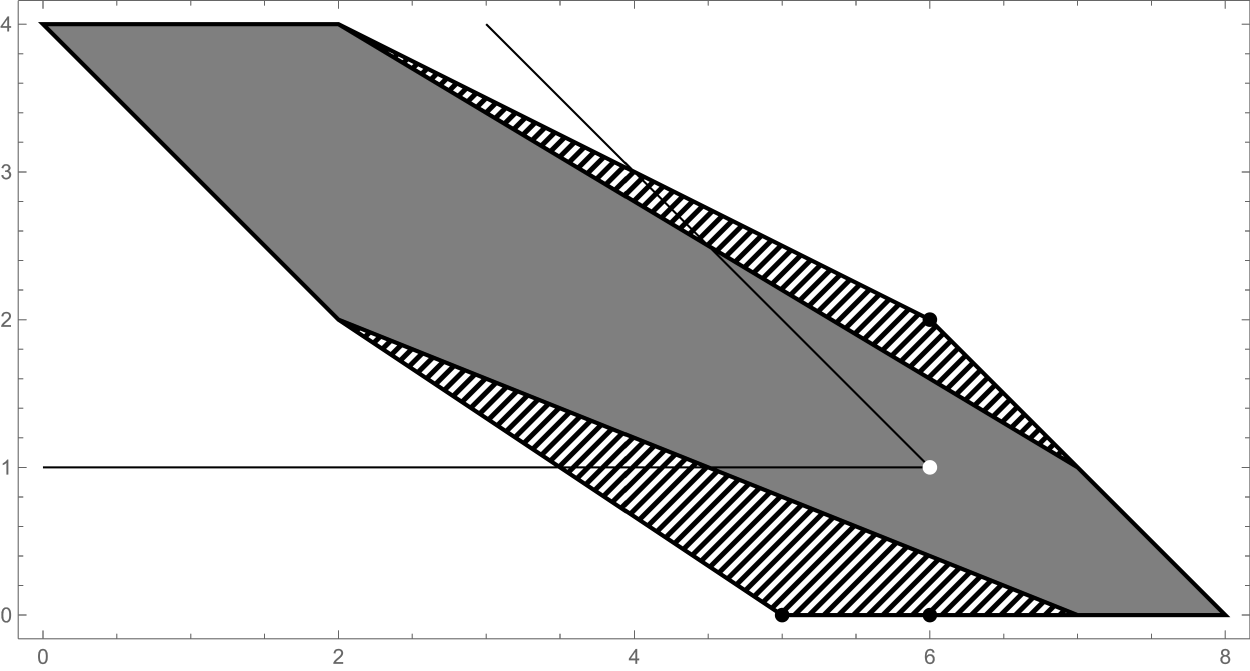}}   
    \caption{The polygon $NP^{\ell,\text{out}}_{m+1}$ contains $NP^{\ell}_{m+1}$ (in solid grey); their difference is the union of two triangles (hatched). The white lattice point inside $NP^{\ell}_{m+1}$ is $v^5_{m+1}$.}
    \label{Newton_Pm_BPm}
\end{figure}

\medskip
To establish that $\text{New}(P^{\ell}_{m+1})$ contains $DP^{\ell}_{m+1}$, we verify the explicit coefficient formulae \eqref{explicit_coefficients_legendre_u=1/2}; to do this, we exploit Proposition~\ref{prop:infl_recursion} as in the proof of Theorem~\ref{generic_support_legendre_newton_polygon}, with slight modifications. The fact that $[v^1_m]P^{\ell}_m$ is merely {\it piecewise} polynomial, for example, reflects the fact that in this case the inductive coefficient relation \eqref{inductive_coefficient_relation} applies for $m \geq 3$, while for $m \in \{1,2\}$ only the second summand on the right-hand side of \eqref{inductive_coefficient_relation} is operative. Similarly, Proposition~\ref{prop:infl_recursion} implies that
\[
\begin{split}
    [x^{2m+1}]P^{\ell}_{m+1}&= \frac{1}{m+1}([x^{2m-2}]D^1 P^{\ell}_m \cdot [x^3]f+ [x^{2m-1}]D^1 P^{\ell}_m \cdot [x^2]f\\
    &+([x^{2m-1}]P^{\ell}_m \cdot [x^2]D^1 f+ [x^{2m}]P^{\ell}_m \cdot [x]D^1 f)\cdot (u-m)\\
    &=\frac{1}{m+1}([x^{2m-1}]P^{\ell}_m \cdot (3u-m-1)- [x^{2m}]P^{\ell}_m \cdot 2u) \\
    &=\frac{3u-m-1}{m+1} \cdot [x^{2m-1}]P^{\ell}_m- \frac{2u(3u)_m}{(m+1)!}
\end{split}
\]
for every $m \geq 1$, and the desired characterization of $[v^3_m]P^{\ell}_m$ now follows easily by induction.

\medskip
We now argue that $[(i,j)]P^{\ell}_{m+1}=0$ for every $(i,j) \in NP^{\ell,\text{out}}_{m+1} \setminus NP^{\ell}_{m+1}$ as follows. The integral lattice points $(i,j) \in NP^{\ell,\text{out}}_{m+1} \setminus NP^{\ell}_{m+1}$ include 
three distinguished points $v^6_{m+1}:=(2m,2)$, $v^7_{m+1}:=(2m,0)$, $v^8_{m+1}:=(2m-1,0)$ that are present for every $m \ge 1$. Additionally, $NP^{\ell,\text{out}}_{m+1} \setminus NP^{\ell}_{m+1}$ contains lattice points $(\frac{3m-1}{2},\frac{m+3}{2})$ and $(\frac{3m-2}{2},1)$ 
whenever $m$ is odd or even respectively, and there is a further interior lattice point $(\frac{3m-1}{2},1)$ of $NP^{\ell,\text{out}}_{m+1} \setminus NP^{\ell}_{m+1}$ whenever $m$ is odd. 
It is not much harder to see that these are in fact the {\it only} lattice points in $NP^{\ell,\text{out}}_{m+1} \setminus NP^{\ell}_{m+1}$. For example, the edge $\overline{v^1_{m+1}v^6_{m+1}}$ of $NP^{\ell,\text{out}}_m$ 
contains no interior lattice points unless $m$ is odd, in which case the midpoint $(\frac{3m-1}{2},\frac{m+3}{2})$ is the unique such point. 
Now $NP^{\ell,\text{out}}_{m+1} \setminus NP^{\ell}_{m+1}$ is the union of the triangles $\text{Conv}(v^1_{m+1},v^6_{m+1},v^4_{m+1})$ and $\text{Conv}(v^2_{m+1},v^3_{m+1},v^8_{m+1})$, and its edges meeting $NP^{\ell}_{m+1}$ have interior lattice points $(\frac{3m}{2},\frac{m+2}{2})$ and $(\frac{3m}{2},1)$ precisely when $m$ is even. It follows from Pick's formula that $\text{Conv}(v^1_{m+1},v^6_{m+1},v^4_{m+1})$ has no interior lattice points for any $m \ge 1$ and that $\text{Conv}(v^2_{m+1},v^3_{m+1},v^8_{m+1})$ has 0 (resp., 1) interior lattice points when $m$ is even (resp., odd); whenever $m$ is odd the point in question must then be $(\frac{3m-1}{2},1)$ as before.

\medskip
Now say that $(i,j) \in \{v^6_{m+1},v^7_{m+1},v^8_{m+1}\}$. 
Proposition~\ref{prop:infl_recursion} together with our induction hypothesis implies that
\[
\begin{split}
[v^6_{m+1}]P^{\ell}_{m+1}&=\frac{1}{m+1}([x^{2m-2}\la]D^1 P^{\ell}_m \cdot [x^2 \la]f 
+[x^{2m-1}\la]P^{\ell}_m \cdot [x \la]D^1 f \cdot (u-m)) \\
&=\frac{1}{m+1}([x^{2m-1}\la]P^{\ell}_m \cdot (1-2u)) 
\end{split}
\]
which vanishes as $u=\frac{1}{2}$. Completely analogously, we have
\[
\begin{split}
[v^7_{m+1}]P^{\ell}_{m+1}&=\frac{1}{m+1}([x^{2m-2}]D^1 P^{\ell}_m \cdot [x^2]f+ [x^{2m-1}]P^{\ell}_m \cdot [x]D^1 f \cdot(u-m)) \\
&= \frac{1}{m+1}[x^{2m-1}]P^{\ell}_m \cdot((2m-1)(-1)- 2(u-m))
\end{split}
\]
which vanishes as $u=\frac{1}{2}$, and
\[
[v^8_{m+1}]P^{\ell}_{m+1}=\frac{1}{m+1}([x^{2m-3}]D^1 P^{\ell}_m \cdot [x^2]f+ [x^{2m-2}]P^{\ell}_m \cdot [x]D^1 f \cdot(u-m))=0
\]
as $[x^{2m-2}]P^{\ell}_m=0$ by induction.

\medskip
Notice that the other lattice points in $NP^{\ell,\text{out}}_{m+1} \setminus NP^{\ell}_{m+1}$ lie inside the union $L^{\ell}_{m+1}$ of lines 
$L^{\ell,1}_{m+1}$ and  $L^{\ell,2}_{m+1}$ of slope $-1$ and 0, respectively. The fact that the corresponding monomials lie outside the support of $P^{\ell}_{m+1}$ will follow from \eqref{explicit_coefficient_v^5_m} and the fact that $[x^i\la^j]P^{\ell}_m=0$ when $(i,j) \in L^{\ell}_m \setminus \{v^5_m\}$ for every $m \geq 2$. Indeed, given any lattice point $(i,j) \in L^{\ell,1}_{m+1}$ with $j >2$, we have the following:\footnote{In fact, this equation holds for {\it any} lattice point $(i,j)$.}
\begin{equation}\label{eqn:newton_recursion_general}
\begin{split}
    [x^i\la^j]P^{\ell}_{m+1}&=\frac{1}{m+1}([x^{i-2}\la^j]D^1 P^{\ell}_m \cdot [x^2]f +[x^{i-3}\la^j]D^1 P^{\ell}_m \cdot [x^3]f\\
    &\phantom{=\frac{1}{m+1}(\;}+[x^{i-1}\la^{j-1}]D^1 P^{\ell}_m \cdot [x\la]f +[x^{i-2}\la^{j-1}]D^1 P^{\ell}_m \cdot [x^2\la]f)\\
    &\phantom{=\;\;}+\frac{u-m}{m+1}([x^{i-1}\la^{j}] P^{\ell}_m \cdot [x]D^1f +[x^{i-2}\la^{j}] P^{\ell}_m \cdot [x^2]D^1f \\
    &\phantom{=+\frac{1}{m+1}(\;\;}+[x^{i}\la^{j-1}] P^{\ell}_m \cdot [\la]D^1f +[x^{i-1}\la^{j-1}] P^{\ell}_m \cdot [x\la]D^1f\\
    &=\frac{1}{m+1}\left([x^{i-1}\la^{j}] P^{\ell}_m \cdot \left(1-i-2u+2m\right)+[x^{i-2}\la^{j}] P^{\ell}_m \cdot \left(i-2+3u-3m\right)\right. \\
    &\phantom{=\frac{1}{m+1}\left(\;\right.}\left.+[x^{i}\la^{j-1}] P^{\ell}_m \cdot \left(i+u-m\right)+[x^{i-1}\la^{j-1}] P^{\ell}_m \cdot \left(1-i-2u+2m\right)\right).
\end{split}
\end{equation}
As $(i-2,j)$ and $(i-1,j-1)$ both belong to $L^{\ell,1}_m \setminus \{v^5_m\}$, it follows that $[x^{i-2}\la^j]P^{\ell}_m=[x^{i-1}\la^{j-1}]P^{\ell}_m=0$ by induction. Furthermore, $[x^{i-1}\la^j]P^{\ell}_m=[x^i\la^{j-1}]P^{\ell}_m=0$ as $j>2$ and $(i-1,j),(i,j-1)$ lie outside $NP^{\ell}_m$ (indeed, they lie in a ray of slope $-1$ whose source is $v^4_m$, 
and whose intersection with $NP^{\ell}_m$ is precisely $\{v^4_m\}$). Given $(2m-1,2) \in L^{\ell,1}_{m+1}$, we have
\[
\begin{split}
    [x^{2m-1}\la^2]P^{\ell}_{m+1}&=\frac{1}{m+1}\left([x^{2m-3}\la^{2}] P^{\ell}_m \cdot \left(3u-m-3\right)+[x^{2m-1}\la^{1}] P^{\ell}_m \cdot \left(u+m-1\right)\right. \\
    &\phantom{=\frac{1}{m+1}\left(\;\right.}\left.+[x^{2m-2}\la^{1}] P^{\ell}_m \cdot \left(2-2u\right)\right)
\end{split}
\]
by \eqref{eqn:newton_recursion_general}.
As $(2m-3,2) \in L^{\ell,1}_m$, $(2m-1,1)=v^4_m$ and $(2m-2,1)=v^5_m$, induction in tandem with \eqref{explicit_coefficients_legendre_u=1/2} and \eqref{explicit_coefficient_v^5_m} now yields
\[
\begin{split}
    [x^{2m-1}\la^2]P^{\ell}_{m+1}&=\frac{1}{m+1}\left(-\frac{2}{(m-1)!}u(3u-1)_{m-1} \cdot \left(u+m-1\right)\right. \\
    &\phantom{=\frac{1}{m+1}\left(\;\right.}\left.+\frac{1}{(m-1)!}((4m-1)u-m)u\cdot (3u-2)_{m-2} \cdot \left(2-2u\right)\right)\\
    &=\frac{u(3u-2)_{m-2}}{(m+1)\cdot (m-1)!} \cdot (((4m-1)u-m)(2-2u)-2(3u-1)(u+m-1))
\end{split}
\]
which is zero 
when $u=\frac12$. Similar arguments to the above show that $[x^i\la]P^{\ell}_{m+1}=0$ for every $(i,1) \in L^{\ell,2}_{m+1} \setminus \{v^5_m\}$. It follows by induction that $\text{New}(P^{\ell}_{m+1}) \subset NP^{\ell}_{m+1}$.

\noindent It remains to prove the formula for $[v^5_{m+1}]P^{\ell}_{m+1}$ given in \eqref{explicit_coefficient_v^5_m}. To wit, by appealing to \eqref{eqn:newton_recursion_general}, we obtain
\[
\begin{split}
    [v^5_{m+1}]P^{\ell}_{m+1} &= \frac{1}{m+1}\left([x^{2m-1}\la] P^{\ell}_m \cdot \left(1-2u\right)+[x^{2m-2}\la] P^{\ell}_m \cdot \left(3u-m-2\right)\right. \\
    &\phantom{=\frac{1}{m+1}\left(\;\right.}\left.+[x^{2m}] P^{\ell}_m \cdot \left(m+u\right) + [x^{2m-1}] P^{\ell}_m \cdot \left(1-2u\right)\right)\\
    &=\frac{1}{m+1}\left(2 \cdot \frac{-2}{(m-1)!}u(3u-1)_{m-1} \cdot (1-2u)+\frac{1}{m!}(3u)_m \cdot \left(m+u\right)\right.\\
    &\phantom{=\frac{1}{m+1}\left(\;\right.}\left.+\frac{1}{(m-1)!}((4m-1)u-m)u\cdot (3u-2)_{m-2} \cdot \left(3u-m-2\right)\right)\\
    &=\frac{u(3u-2)_{m-2}}{(m+1)!}(-4m(3u-1)(1-2u)+m((4m-1)u-m)(3u-m-2)\\
    &\phantom{=\frac{u(3u-2)_{m-2}}{(m+1)!}(\;}+3(3u-1)(m+u))\\
    &=\frac{u(3u-2)_{m-2}}{(m+1) \cdot m!}((m+1)((4m+3)u-m-1)(3u-m))\\
    &=\frac{1}{m!}((4(m+1)-1)u-(m+1))u \cdot (3u-2)_{(m+1)-2}
\end{split}
\]
which proves \eqref{explicit_coefficient_v^5_m}.
\end{proof}

\subsection{Singularities and genera of superelliptic Weierstrass inflectionary curves}\label{sec:weierstrass_inflectionary_curves}
To close this section, we characterize the Newton polygons of atomic inflectionary curves derived from the superelliptic Weierstrass family $y^n= x^3+ \la x+2$ when $u=\frac{1}{2}$. 
We first characterize those polygons associated with the linear change of variables $(x \mapsto x+1, \la \mapsto \la-3)$ that translates the origin $(0,0) \in \mb{A}^2_{x,\la}$ to the singular point $(1,-3)$.

\begin{thm}\label{weierstrass_Newton_polygon}
Suppose that $n=2\ell$ and that $\mathrm{char}(F)$ is either zero or sufficiently positive. 
For every positive integer $m \geq 3$, the Newton polygon of the inflection polynomial $P^{\ell}_m$ derived from $y^n=x^3+ \la x+ 2$ with respect to affine coordinates centered in $(x=1,\la=-3)$ is
\[
{\rm New}(P^{\ell}_m)=\text{Conv}((0, \lceil m/2 \rceil),(0,m),\de_{2|(m-1)}(1,(m-1)/2), (m-2,1),(2m-1,0),(2m,0))
\]
in which $\de_{2|(m-1)}$ indicates that this vertex is only operative when $m$ is odd.
\end{thm}

\begin{proof}
We adopt the same basic strategy used in the proof of Theorem~\ref{support_legendre_newton_polygon_z=2}. We let $P^{\ell,\ast}_m=P^{\ell,\ast}_m(x,\la)$ denote the polynomial obtained from $P^{\ell}_m$ upon substituting $(x \mapsto x+1, \la \mapsto \la-3)$; equivalently, this is the $(\ell,m)$-th atomic inflection polynomial associated to the polynomial $f^{\ast}=x^3+ 3x^2+ \la x+ \la$ obtained from $f=x^3+ \la x+ 2$ via the same change of coordinates. Set
$v^1_m=(0, \lceil m/2 \rceil), v^2_m=(0,m), v^3_m=(1,\frac{m-1}{2}), v^4_m=(m-2,1), v^5_m=(2m-1,0)$, and $v^6_m=(2m,0)$. We claim that
\begin{equation}\label{eq:v^i_m_weierstrass_coeffs}
\begin{split}
[v^1_m]P^{\ell,\ast}_m&=\bigg(\frac{3^{1+\de_{2|m}}}{2}\bigg)^{\de_{m> 3}}(3u-m+1)^{\de_{2 | (m-1)}} \cdot (u)_{\lfloor m/2 \rfloor}, [v^2_m]P^{\ell,\ast}_m= \frac{1}{m!} (u)_m, \\ 
[v^3_m]P^{\ell,\ast}_m&= \frac{2 \cdot 3^{\frac{m+1}{2}}}{(\frac{m-1}{2})!} \cdot (u)_{\frac{m+1}{2}}, [v^4_m]P^{\ell,\ast}_m=\frac{3^{m-1} \cdot 2^{\de_{2|(m-1)}}}{(\lfloor m/2 \rfloor-1)!(\!(3)\!)^{\lfloor m/2\rfloor-1}}(\!(2u-3)\!)_{\lfloor m/2 \rfloor-1}(u)_{\lceil m/2 \rceil},\\
[v^5_m]P^{\ell,\ast}_m&=\frac{1}{(3)^{m-3}}(3u)_m, \text{ and }
[v^6_m]P^{\ell,\ast}_m=\frac{1}{m!}(3u)_m
\end{split}
\end{equation}
for every integer $m \geq 3$ and every $u \in (0,1)$; and that $[(i,j)]P^{\ell,\ast}_m=0$ for every $(i,j) \notin \text{Conv}(\{v^k_m\}_{k=1}^6)$. Here 
$\de$ is Kronecker's delta. It is easy to check that our claims hold when $m=3$ and $m=4$; arguing inductively, assume they hold for (every index less than or equal to) some $m \geq 4$.
Now say $m$ is even. Applying Proposition~\ref{prop:infl_recursion} in tandem with our inductive hypothesis, we then have $\text{New}(P^{\ell,\ast}_{m+1}) \sub \text{NP}^{\ell,\ast; \text{out}}_{m+1}$, where
\[
\begin{split}
\text{NP}^{\ell,\ast; \text{out}}_{m+1} &:=\text{Conv}(\text{New}(D^1 P^{\ell,\ast}_{m}) \oplus_M \text{New}(f^{\ast}) \bigcup \text{New}(P^{\ell,\ast}_{m}) \oplus_M \text{New}(D^1 f^{\ast})) \\
&=\text{Conv}((0,m/2+1),(0,m+1),(1,m/2),(m-1,1),(2m,0),(2m+2,0)).
\end{split}
\]
The difference between $\text{NP}^{\ell,\ast; \text{out}}_{m+1}$ and the polygon that we claim is $\text{New}(P^{\ell,\ast}_{m+1})$ is the lattice triangle 
\[
\De=\text{Conv}((m-1,1),(2m,0),(2m+1,0)).
\]
Here $\De$ is of area $\frac{1}{2}$; it follows from Pick's theorem that $\De$ has no interior lattice points, and that its only boundary lattice points are its vertices. Among these, only $(2m,0)$ lies outside the we claim is $\text{New}(P^{\ell,\ast}_{m+1})$. 
But Proposition~\ref{prop:infl_recursion} together with our inductive hypothesis and the fact that $u=\frac{1}{2}$ imply that 
\[
\begin{split}
[(2m,0)]P^{\ell,\ast}_{m+1}&= \frac{1}{m+1}([(2m-2,0)]D^1 P^{\ell,\ast}_m \cdot [(2,0)]f^{\ast}+ [(2m-1,0)]P^{\ell,\ast}_m \cdot [(1,0)]D^1 f^{\ast} \cdot (u-m)) \\
&=\frac{1}{m+1}((2m-1)+2(u-m)) \cdot [(2m-1,0)]P^{\ell,\ast}_m \cdot [(2,0)]f^{\ast} \\
&=0.
\end{split}
\]
A nearly-identical argument works when $m$ is odd. Namely, setting 
\[
\text{NP}^{\ell,\ast; \text{out}}_{m+1} :=\text{Conv}(\text{New}(D^1 P^{\ell,\ast}_{m}) \oplus_M \text{New}(f^{\ast}) \bigcup \text{New}(P^{\ell,\ast}_{m}) \oplus_M \text{New}(D^1 f^{\ast}))
\]
as before, the difference between $\text{NP}^{\ell,\ast; \text{out}}_{m+1}$ and the polygon that we claim is $\text{New}(P^{\ell,\ast}_{m+1})$ is precisely 
\[
\De=\text{Conv}((m-1,1),(2m,0),(2m+1,0)).
\]
We leave the slightly tedious, but straightforward inductive verification of our explicit formulae \eqref{eq:v^i_m_weierstrass_coeffs} for $v^i_m[P^{\ell,\ast}_m]$, $i=1,\dots,6$ to the reader.
\end{proof}

As we will now explain, the topological type of the singularity of $\mc{C}^{\ell}_m$ in $(1,-3)$ is in fact determined by its associated {\it local Newton polygon}; that is, by the lower hull of the Newton polygon in Theorem~\ref{weierstrass_Newton_polygon}.
Whenever $m$ is greater than 5, this local Newton polygon consists of two (resp., three) segments when $m$ is even (resp., odd),  one of which   contains lattice points other than its vertices. Specifically, when $m$ is even (resp., odd), the edge linking $v^1_m$ (resp., $v^3_m$) and $v^4_m$ contains lattice points $(2j, \frac{m}{2}-j)$, $j=1,\dots,\frac{m}{2}-2$ (resp., $(1+2j, \frac{m-1}{2}-j)$, $j=1,\dots, \frac{m-1}{2}-2)$); see Figure \ref{LNP} below.

\begin{figure}[htbp]
\begin{center}
\includegraphics{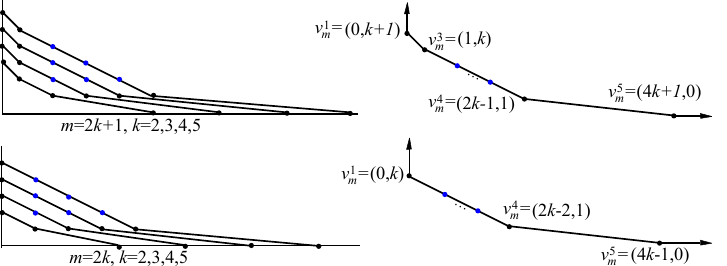}
\caption{Local Newton polygons of the plane curve singularity in $(1,-3)$ of $\mc{C}^{\ell}_m$}
\label{LNP}
\end{center}
\end{figure}

\noindent The corresponding coefficients 
of $P^{\ell,\ast}_m$ appear to always split into explicitly identifiable $u$-linear factors.

\begin{conj}\label{inner_edge_conjecture}
Suppose $n=2\ell$ and that $\text{char}(F)$ is either zero or sufficiently positive. For every even positive integer $m=2k$ with $k \geq 3$ the atomic inflection polynomial $P^{\ell,\ast}_m$ derived from $y^n=x^3+ \la x+ 2$ (and adapted to coordinates centered in $(1,-3)$) satisfies
\[
[(2j,k-j)]P^{\ell,\ast}_m= c_{j,k} \cdot (u)_k (\!(2u-2k+1)\!)^j
\]
for every $j=1,\dots,k-2$, where 
$c_{j,k}=\frac{3^{j+k} (2j+1)}{(k-j)!\prod_{i=1}^ji(2i+1)}$. 
Similarly, for every odd positive integer $m=2k+1$ with $k \geq 3$, we have
\[
[(2j+1,k-j)]P^{\ell,\ast}_m= d_{j,k} \cdot (u)_{k+1} (\!(2u-2k+1)\!)^j
\]
for every $j=1,\dots,k-2$, where 
$d_{j,k}=\frac{2 \cdot 3^{j+k+1}}{(k-j)!\prod_{i=1}^ji(2i+1)}$.
\end{conj}

Conjecture~\ref{inner_edge_conjecture} predicts that whenever $k\geq 2$ and $u=\frac{1}{2}$, the inflectionary curve $\mc{C}^{\ell}_{2k}$ (resp., $\mc{C}^{\ell}_{2k+1}$) has a singularity at $(1,-3)$ with local normal form $x^{4k-1}+ \sum_{j=0}^{k-1} \al_j x^{2j} \la^{k-j}=0$ (resp., $x^{4k+1}+ \sum_{j=0}^{k-1} \al_j x^{2j+1} \la^{k-j}+ \be \la^{k+1}=0$), where the $\al_j$, $j=1,\dots,k-1$ and $\be$ are nonzero scalars. In order to derive their topological types, we will make use of the following two notions from singularity theory.

\begin{dfn}
A polynomial $f$ in two variables is \emph{quasi-homogeneous} whenever its Newton polygon $\text{New}(f)$ is a segment; the affine curve $V(f)\subset(\mathbb C^*)^2$ it defines is a \emph{quasi-line} whenever $New(f)$
is a segment of lattice length 1.
\end{dfn}

\begin{dfn}
Given a quasi-homogeneous polynomial $f$ with Newton polygon of lattice length $\ell$, 
we say that $f$ is \emph{Newton non-degenerate} whenever $V(f)\subset(\mathbb C^*)^2$ consists of $\ell$ distinct quasi-lines.
\end{dfn}

According to \cite[p. 226]{BT}, a quasi-homogeneous polynomial $f$ is Newton non-degenerate whenever it contains no repeated irreducible factors. In our case, this means that the singularity of $\mc{C}^{\ell}_m$ in $(1,-3)$ is Newton non-degenerate provided the restriction $P_m^{\ell,\ast}|_{[v^1_m,v^4_m]}$ when $m$ is even (resp., $P_m^{\ell,\ast}|_{[v^3_m,v^4_m]}$ when $m$ is odd) contains no repeated irreducible factors. This, in turn, is equivalent to the specializations of each of these polynomials in $x=1$ being {\it separable}, viewed as polynomials in $\la$.


\begin{rem}\label{rem:weierstrass_lower_hull}
\emph{
Given positive integers $k\geq2$ and $1\leq j\leq k-1$, let 
\[
\gamma_{j,k}(u)=\frac{2^j\cdot3^{j}}{(k-j)!\prod_{i=1}^ji(2i+1)}\prod_{i=1}^j(u-(k-i+\tfrac{1}{2})).
\]
Theorem ~\ref{weierstrass_Newton_polygon} and Conjecture~\ref{inner_edge_conjecture} together
predict the following.
\begin{itemize}
\item For every $m=2k+1$, the polynomial  $\tfrac{1}{[v^3_m]P_m^{\ell,*}}\la^{-1}x^{-1}P_m^{\ell,*}|_{[v_m^3,v_m^4]}$ has coefficients $\{k!\gamma_{j,k}(u):j=1,\ldots,k-1\}$; and its irreducible factors correspond to those of
\[
Q_{k,\text{odd}}(\lambda):=
\lambda^{k-1}+k!\sum_{j=1}^{k-1}\gamma_{j,k}(u)\lambda^{k-1-j}.
\]
\item For every $m=2k$, the polynomial  $\tfrac{1}{[v^1_m]P_m^{\ell,*}}\la^{-1}P_m^{\ell,*}|_{[v_m^1,v_m^4]}$ has coefficients $\{2\cdot3^{k-2}(2j+1)\gamma_{j,k}(u):j=1,\ldots,k-2\}\cup\{2\cdot3^{k-2}\gamma_{k-1,k}(u)\}$; and its irreducible factors correspond to those of
\[
Q_{k,\text{even}}(\lambda):=
\lambda^{k-1}+3^{k-2}\cdot2\biggl(\sum_{j=1}^{k-2} (2j+1)\gamma_{j,k}(u)\lambda^{k-1-j}+\gamma_{k-1,k}(u)\biggr).
\]
\end{itemize}}
\end{rem}

Newton non-degenerate singularities have embedded toric resolutions that depend only on their underlying Newton polygons.
To spell out a resolution explicitly, we first fix a regular refinement $\Sigma_m$ of the Newton fan of the local Newton polygon. 
According to \cite[Prop. 5.1]{BT}, there is a neighborhood $U$ of the origin in $\mb{A}^2$ for which the strict transform of $X_m\cap U$
under the toric map $\pi(\Sigma_m):\text{Tor}(\Sigma)\to \mb{A}^2$ is non-singular (and transversal in each chart with respect to the strata of the canonical stratification).

\medskip
In our case, the Newton fans are as in Figure \ref{NF}a and \ref{NF}b for odd $m\geq5$ and even $m\geq6$, respectively; and $\Sigma_m$ is the fan determined by the collections of vectors $\{\beta_i=(1,i):i=0,\ldots,m+1\}\cup\{\beta_{m+2}=(0,1)\}$. Note that $\text{det}(\beta_i,\beta_{i+1})=1$ for every $i=0,\ldots,m+1$.

\begin{figure}[htbp]
\begin{center}
\includegraphics{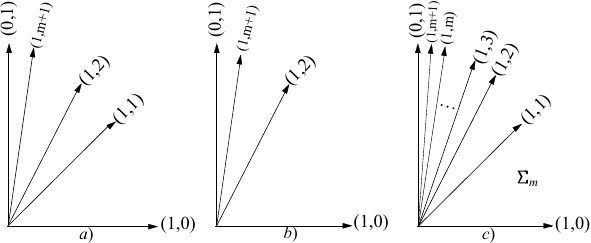}
\caption{The Newton fan of the curve germ $X_m$ for a) odd indices $m\geq 5$ and b) even indices $m\geq 4$; and c) the regular refinement $\Sigma_m$.}
\label{NF}
\end{center}
\end{figure}

\begin{conj}\label{conj:separability}
Suppose $n=2\ell$ and that $\text{char}(F)$ is either zero or sufficiently positive. The restrictions $P^{\ell,*}_m|_{[v_m^1,v_m^4]}$ when $m$ is even (respectively $P^{\ell,*}_m|_{[v_m^3,v_m^4]}$ when $m$ is odd) are Newton non-degenerate for every positive integer $m \geq 6$. 
\end{conj}

 The upshot of Conjecture~\ref{conj:separability}, assuming it holds, is that the Weierstrass inflectionary curve $\mc{C}^{\ell}_m$ has Newton non-degenerate singularities in $(1,-3)$ and its images under the $\mu_3$-action whenever $m \geq 3$ and $u=\frac{1}{2}$. Taken together with Theorem~\ref{weierstrass_Newton_polygon}, which shows that the quasi-lines indexing the components of the singularity in $(1,-3)$ have normal forms $\la+ \al x^{\be}$ with $\al \in F$ and $\be \in \mb{N}$ and are therefore {\it smooth}, we conclude that the singularity in $(1,-3)$ is topologically a planar multiple point. Non-degeneracy may be decided by computing the resultants $\text{res}_{\la}(Q_{k,\text{even}}(\la), D_{\la} Q_{k,\text{even}(\la)})$ and $\text{res}_{\la}(Q_{k,\text{odd}}(\la), D_{\la} Q_{k,\text{odd}}(\la))$ with $k= \lfloor \frac{m}{2}\rfloor$, the first few of which we list below. All are nonzero in $u=\frac{1}{2}$, which confirms non-degeneracy in these cases.

\begin{center}
\scalebox{0.8}{
    \begin{tabular}{|p{0.2cm}|p{14.5cm}|}
    \hline
        $m$& Resultant\\
     \hline
     6&$(2u - 5)(82u - 213)$\\
     \hline
     7&$(2u - 5)(2u - 13)$\\
     \hline
     8&$(8648u^3 - 99644u^2 + 366558u - 433225)(2u - 5)(2u - 7)^2$\\
     \hline
     9&$(1544u^3 - 4124u^2 - 68050u + 261375)(2u - 5)(2u - 7)^2$\\
     \hline
     10&{\tiny$(2628587072u^6 - 119949472448u^5 + 2150917889200u^4 - 19208897405344u^3 + 88953911319420u^2 -202718213505900u + 178829173396125)(2u - 5)(2u - 7)^2(2u - 9)^3$}\\
     \hline
     11&{\tiny $(73280u^6 - 1800896u^5 + 17586352u^4 - 79585696u^3 + 105411708u^2 + 385941780u - 1128308643)(2u - 5)(2u - 7)^2(2u - 9)^3$}\\
     \hline
     12&{\tiny $(122191605826942938112u^{10} - 5137275780237419929600u^9 + 95483958308251060967680u^8 -
-1028332887864338872274944u^7 + 7059380175502383351849856u^6 - 31945737444679130293915648u^5 +
+94800566724756623412919584u^4 - 175771544931032155005282048u^3 + 177787778371141570008211548u^2-
-638533704751862399285280u - 27254076392882562131835675)(2u - 5)(2u - 7)^2(2u - 9)^3(2u - 11)^4$}\\
\hline
13&{\tiny $(578660864u^9 - 24160546560u^8 + 442054845440u^7 - 4661843030528u^6 + 31203235602752u^5 - 
-135392606249696u^4 + 356010098185728u^3 - 390059717289536u^2 - 495186636360654u + 1452343719158325)
(2u - 5)(2u - 7)^2(2u - 9)^3(2u - 11)^4(2u - 25)$}\\
\hline
14&{\tiny $(517054051760584040013824 u^{15} - 34606335211379129061806080 u^{14} + 1074564038131661482964643840 u^{13} -
-20553014285527963635148863488 u^{12} + 271168279767820327841178212864 u^{11} - 2618660106209969893672606463744 u^{10} +
+19159417066255333643901464918528 u^9 - 108348468080864190587659844395520 u^8 + 477991778387823949401622023180192 u^7 -
-1643879633811405905416224201435376 u^6 + 4355137706111155294069745610578032 u^5 - 8656637453479899643766490173287192 u^4 +
+12306388582199883590702872639926470 u^3 - 11486088155884566051550904286897225 u^2 + 5933971159136931086501205617667000 u -
-1070586531538239512727665046860625) (2 u - 5) (2 u - 7)^2 (2 u - 9)^3 (2 u - 11)^4 (2 u - 13)^5$}\\
\hline
15&{\tiny $(3582640383754240 u^{15} - 10728692298111500288 u^{14} + 875818844352211918848 u^{13} - 
-33076720356213968580608 u^{12} + 759534115415560817821696 u^{11} - 11819438460454402630634496 u^{10} + 
+131714774285747089485097472 u^9 - 1083553404957176512706490112 u^8 + 6684602789076604465188623232 u^7 -
-31057766997114205392258042688 u^6 + 107656638625140000720091457888 u^5 - 268810265560488302494068274704 u^4 +
+436139961317292156103910305944 u^3 - 300075462388081764565510962564 u^2 - 324321675886222418719532454930 u +
+645480864299668165786600310475) (2 u - 5) (2 u - 7)^2 (2 u - 9)^3 (2 u - 11)^4 (2 u - 13)^5$}\\
\hline
    \end{tabular}}
\end{center}

\begin{thm}\label{weierstrass_infl_curve_autom}
Suppose that $n=2\ell$. 
For every positive integer $m \geq 3$, the inflectionary curve $\mc{C}^{\ell}_m$ derived from $y^n=x^3+\la x+2$ is equipped with a $\mu_3$-symmetry given by $(x \mapsto g x,\la \mapsto g^{-1} \la)$, where $g \in \mu_3 \subset \mb{G}_m$ is an element of $\mu_3$. 
\end{thm}

\begin{proof}
We have $P^{\ell}_3= 2-\frac{5}{2}x^3- \frac{1}{16}x^6+ \frac{1}{2}x \la- \frac{5}{16} x^4 \la+ \frac{5}{16} x^2 \la^2+ \frac{1}{16}\la^3$, so the desired result clearly holds when $m=3$. Likewise $f(x)=x^3+\la x+2$ is left invariant by the $\mu_3$-action, while $g \in \mu_3$ acts on $D^1_x f= 3x^2+ \la$ by multiplying by $g^{-1}$. We now argue inductively, and assume that $P^{\ell}_m$ is multiplied by $g^j$ for some $j \in \{0,1,2\}$ by the $\mu_3$-action. In view of Proposition~\ref{prop:infl_recursion}, it suffices to show that $D^1_x P^{\ell}_m$ is multiplied by $g^{j-1}$; but this is clear.
\end{proof}

\begin{rem}\label{rem:weierstrass_behavior_in_low_char}
\emph{It is natural to wonder about the dependence of the Newton polygons of superelliptic inflection polynomials $P^{\ell}_m$ (in distinguished choices of local coordinates) on the dependence of the characteristic of the underlying base field $F$ when $F$ is positive yet arbitrary. Remark~\ref{rmk:infl_recursion_ext}, coupled with the proof of Theorem~\ref{generic_support_legendre_newton_polygon}, implies that the expression $\frac{(au)_n}{n!}$ is well-defined in $\mb{F}_p$ whenever $(au)_n$ is nonvanishing, for otherwise arbitrary choices of positive integers $a$ and $n$, $u \in \mb{Q} \cap (0,1)$, and every prime integer $p$ relatively prime to the degree of the underlying superelliptic covers. The $p$-adic valuation $\text{val}_p(\frac{(au)_n}{n!})$, in turn, affects the structure of the Newton polygons $\text{New}(P^{\ell}_m)$ that arise from the specializations of Legendre and Weierstrass pencils over $\mb{Z}$ to $\mb{F}_p$. In particular, when $u=\frac{1}{2}$ and $(au)_n$ is nonvanishing, we have
\[
\text{val}_p \bigg(\frac{(au)_n}{n!}\bigg)= \text{val}_p ((a/2)_n)- \text{val}_p(n!)=  \text{val}_p ((\!(a)\!)_n)- \text{val}_p(n!)
\]
for every odd prime $p$. A classical theorem of Legendre establishes, moreover, that $\text{val}_p(n!)=\sum_{i=1}^{\infty} \lfloor \frac{n}{p^i}\rfloor$ for every $n$. Now suppose that $a>2n-2$; then either $a$ is even, in which case Legendre implies that
\[
\text{val}_p ((\!(a)\!)_n)= \text{val}_p(a/2)_n= \sum_{i=1}^{\infty} \lfloor \frac{a/2}{p^i}\rfloor - \sum_{i=1}^{\infty} \lfloor \frac{(a/2-n)}{p^i}\rfloor;
\]
or else $a$ is odd, in which case Legendre yields
\[
\begin{split}
\text{val}_p ((\!(a)\!)_n)&= \text{val}_p(a!)- \text{val}_p((a-2n+1)!)- \text{val}_p((\!(a-1)\!)_{n-1}) \\
&= \sum_{i=1}^{\infty} \lfloor \frac{a}{p^i}\rfloor -\sum_{i=1}^{\infty} \lfloor \frac{(a-2n+1)}{p^i}\rfloor- \sum_{i=1}^{\infty} \lfloor \frac{(a-1)/2}{p^i}\rfloor+ \sum_{i=1}^{\infty} \lfloor \frac{(a-1)/2-n+1}{p^i}\rfloor.
\end{split}
\]
Similarly, if $u=\frac{1}{2}$ and $a<2n-2$, then nonvanishing of $(au)_n$ means that $a$ is necessarily odd, and
\[
\text{val}_p((\!(a)\!)_n)=\sum_{i=1}^{\infty} \lfloor \frac{a}{p^i} \rfloor- \sum_{i=1}^{\infty} \lfloor \frac{(a-1)/2}{p^i} \rfloor+ \sum_{i=1}^{\infty} \lfloor \frac{(2n-2-a)}{p^i} \rfloor- \sum_{i=1}^{\infty} \lfloor \frac{(n-1-a/2)}{p^i} \rfloor.
\]
}
\end{rem}

Note that when $\mathrm{char}(F) \neq 3$, $\mu_3 \cong \mb{Z}/3\mb{Z}$ generated by a primitive cube root $\zeta$ of unity, and the action on $\mc{C}^{\ell}_m \sub \mb{A}^2_{x,\la}$ extends to a linear action on a {\it weighted} projective space $\mb{P}(1,2,1)$ given by $\zeta \cdot [x:\la:z]=[\zeta x:\zeta^{-1} \la: z]$. 
An upshot of Theorem~\ref{weierstrass_infl_curve_autom} is that for every $m \geq 3$, $\mc{C}^{\ell}_m$ has isomorphic singularities in $(\zeta^{-j},-3\zeta^j)$, $j \in \{0,1,2\}$. 
Moreover, an easy inductive argument using Proposition~\ref{prop:infl_recursion} shows that the ``usual" Newton polygon of $P^{\ell}_m$ in coordinates $x,\la$ lies inside the lattice simplex with vertices $(0,0)$, $(2m,0)$, and $(0,m)$, and always includes $(2m,0)$ and $(0,m)$. It follows that $\mc{C}^{\ell}_m$ may be compactified inside $\mb{P}(1,2,1)$, and doing so introduces no additional singularities at torus-fixed points of the line at infinity ($z=0$), while compactifying $\mc{C}^{\ell}_m$ inside $\mb{P}^2$ introduces a singularity at $[0:1:0]$, which is a torus fixed point.

\medskip
On the other hand, when $\mathrm{char}(F)=3$, 
we have $\mu_3 \cong F[t]/(t^3-1) \cong F[t]/(t-1)^3$, a non-reduced group scheme.\footnote{In this case, $n$ cannot be divisible by 3 by assumption.} Since the Weierstrass family $y^n=x^3+\la x+2$ is defined over $\mb{F}_3$, the same is true of $\mc{C}^{\ell}_m$ for every $m$; and correspondingly $\mc{C}^{\ell}_m$ over $F$ is 
obtained from $\mc{C}^{\ell}_m$ over $\mb{F}_3$ via the base change induced by the natural map $\mathrm{Spec}\: F \rightarrow \mathrm{Spec}\: \mb{F}_3$. So assume that $F \cong \mb{F}_3$. Theorem~\ref{weierstrass_infl_curve_autom} then implies that for every $m \geq 3$, $\mc{C}^{\ell}_m$ has a singularity at $(1,0)$, and admits a compactification inside $\mb{P}(1,2,1)$. 
However ``extra" singularities appear when $m>3$. 
Indeed, over $\mb{Z}[\frac{1}{2}]$ we have
\begin{align*}
    P^{\ell}_3 = &2-\frac{5}{2}x^3-\frac{1}{16}x^6+\frac{1}{2}x\la-\frac{5}{16}x^4\la+\frac{5}{16}x^2\la^2+\frac{1}{16}\la^3,\\
    P^{\ell}_4 = &-\frac{15}{2}x^2+\frac{21}{8}x^5+\frac{3}{128}x^8-\la-\frac{7}{4}x^3\la+\frac{7}{32}x^6\la+\frac{1}{8}x\la^2-\frac{35}{64}x^4\la^2-\frac{5}{32}x^2\la^3-\frac{5}{128}\la^4, \text{ and}\\
    P^{\ell}_5 =&-6x+ 18x^4- \frac{45}{16}x^7- \frac{3}{256}x^{10}+ \frac{9}{4}x^2 \la+ \frac{63}{16}x^5 \la- \frac{45}{256} x^8 \la+ \frac{3}{4}\la^2- \frac{15}{16} x^3 \la^2+ \frac{105}{128}x^6 \la^2\\
    &-\frac{3}{16}x\la^3+ \frac{27}{128}x^4 \la^3+ \frac{33}{256}x^2 \la^4+ \frac{7}{256}\la^5.
\end{align*}
Reducing coefficients modulo 3, we see that $\mc{C}^{\ell}_4$ is reducible, while $\mc{C}^{\ell}_5$ is {\it non-reduced}.


\begin{conj}\label{conj:weierstrass_sing}
Suppose that $n=2\ell$ and that $\text{char}(F)$ is either zero or sufficiently positive (so that it is not three). For every positive integer $m \geq 3$, the inflectionary curve $\mc{C}^{\ell}_m \sub \mb{P}(1,2,1)$ derived from $y^n=x^3+ \la x+ 2$ is nonsingular away from $(\zeta^{-j},-3\zeta^j,1)$, $j \in \{0,1,2\}$, where $\zeta$ is a primitive cube root of unity. 
\end{conj}
The fact that the points $[\zeta^{-j}:-3\zeta^j:1]$, $j \in \{0,1,2\}$ appear as (supports of) singularities of the Weierstrass inflectionary curves $\mc{C}^{\ell}_m$ is unsurprising. Namely, the $\la$-coordinates $-3\zeta^j$ comprise the roots of the $x$-discriminant $-4(27+\la^3)$ of $f(x,\la)=x^3+ \la x+ 2$, and as such index the three singular fibers of the Weierstrass pencil; the $x$-coordinates $\zeta^{-j}$ are the $x$-coordinates of the corresponding singularities. It is natural to expect that this phenomenon persists more generally, and we will return to this point in the following section. On the other hand, the {\it delta-invariants} of singularities of any complete curve embedded in a toric surface are determined by the corresponding Newton polygons. The following result is the key operative ingredient.

\begin{thm}\label{thm:arith_genus_via_newton_polygon}
Let $\iota: X \hra Y$ denote the embedding of an irreducible projective curve embedded in a normal projective toric surface $Y=\text{Tor}(\De)$ over a field $F$, and assume that $\iota(X) \cap \mathrm{Sing}(Y)= \emptyset$. 
The arithmetic genus of $X$ is equal to the number of interior lattice points in the Newton polygon associated to $\iota$.
\end{thm}

\begin{proof}
When $Y$ is smooth, this follows immediately from \cite[Lem. 3.4]{KWZ}; their argument shows that the interior lattice points in the Newton polygon of $\iota$ index a basis of $H^0(Y, K_Y+X)$. In our case, since $\iota (X) \cap \mathrm{Sing}(Y) = \emptyset$, $\iota$ extends to an embedding $\overline{\iota}:X \hookrightarrow \overline{Y}$ with $\overline\iota(X) \cap \mathrm{Sing}(\overline Y)= \emptyset$, where $\overline{Y} \rightarrow Y$ is a toric resolution of singularities. Replacing $Y$ by $\ov{Y}$, we now conclude by applying {\it loc. cit.} once more.
\end{proof}

\begin{conj}\label{conj:weierstrass_geom_genus}
Suppose that $n=2\ell$ and that $\text{char}(F)$ is either zero or sufficiently positive. 
For every positive integer $m \geq 3$, the inflectionary curve $\mc{C}^{\ell}_m \sub \mb{P}(1,2,1)$ derived from $y^n=x^3+ \la x+ 2$ is geometrically irreducible, and of geometric genus $\lceil \frac{(m-1)^2}{4} \rceil$. 
\end{conj}

Indeed, according to Theorem~\ref{thm:arith_genus_via_newton_polygon}, the arithmetic genus of $\mc{C}^{\ell}_m \sub \mb{P}(1,2,1)$ is equal to the number of interior lattice points of the lattice simplex with side lengths $m$, $m$, and $2m$; and this is precisely $(m-1)^2$. On the other hand, the delta-invariant of each of the three isomorphic singularities of $\mc{C}^{\ell}_m \sub \mb{P}(1,2,1)$ described in Theorem~\ref{weierstrass_Newton_polygon} is equal to $(\frac{m-1}{2})^2$ (resp., $\frac{m}{2}(\frac{m}{2}-1)$) when $m$ is odd (resp., even), as this is precisely the number of interior lattice points ``excluded" by the lower hull of the corresponding Newton polygon.

\medskip
It is worth noting here that conjectures \ref{conj:weierstrass_sing} and \ref{conj:weierstrass_geom_genus} (along with the other conjectures in this paper) 
are true whenever $m$ is small. In particular, $\mc{C}^{\ell}_3$ has an {\it elliptic} normalization whenever $n=2\ell$ (here, we will abusely use $\mc{C}^{\ell}_3$ to denote the compactification of the affine inflectionary curve in $\mb{P}(1,2,1)$). Likewise, given that $\mc{C}^{\ell}_3$ admits a $\mu_3$-action, it is natural to try identifying its $\mu_3$-quotient $\mc{Q}^{\ell}_3$.

\begin{prop}\label{prop:Cell3_quotient_one}
    Whenever $n=2\ell$, $F$ is perfect, and $\mathrm{char}(F) \notin \{2,3\}$, the $\mu_3$-quotient $\mc{Q}^{\ell}_3$ 
    of $\mc{C}^{\ell}_3$ is then $F$-isomorphic to a nodal plane cubic curve with an $F$-rational smooth point.
\end{prop}

\begin{proof}
We first claim that $\mc{Q}^{\ell}_3$ has geometric genus zero. To prove the claim, we will apply the Riemann-Hurwitz formula to the natural ($\mu_3$-quotient) map from the normalization of $\mc{C}^{\ell}_3$ to that of $\mc{Q}^{\ell}_3$. More precisely, we start from the following commutative diagram over $F$, where the horizontal morphisms are normalizations and vertical morphisms are $\mu_3$-quotients:
\[
	\begin{tikzcd}
        \mc{C}^{\ell,\nu}_3 \arrow[r] \arrow[d] & \mc{C}^{\ell}_3 \arrow[d] \\
        \mc{Q}^{\ell,\nu}_3 \arrow[r] & \mc{Q}^{\ell}_3
	\end{tikzcd}
\]
Note that as $\mathrm{char}(F)=0$ or $\mathrm{char}(F)>3$, the $\mu_3$-action is separable. Further, since $\mc{C}^{\ell}_3$ is irreducible in $\mb{P}(1,2,1)$, $\mc{C}^{\ell,\nu}_3$ must be a smooth curve of genus equal to the geometric genus of $\mc{C}^{\ell}_3$. The $\mu_3$-action on $\mc{C}^{\ell}_3$ induces a natural $\mu_3$-action on $\mc{C}^{\ell,\nu}_3$ as well, and the $\mu_3$-fixed points of $\mc{C}^{\ell,\nu}_3$ are preimages of those of $\mc{C}^{\ell}_3$. To locate the $\mu_3$-fixed points of $\mc{C}^{\ell}_3$, note that the action is in fact induced by the $\mu_3$-action on the ambient $\mb{P}(1,2,1)$, and the $\mu_3$-fixed points of $\mb{P}(1,2,1)$ consist of $[0:0:1]$ and the line at infinity $(z=0)$. Thus the $\mu_3$-fixed points of $\mc{C}^{\ell}_3$ all lie along $(z=0)$, and homogenizing $P^{\ell}_3$ with respect to $z$ (as a degree 1 variable) and substituting $z=0$ yields
\[
    -\frac{1}{16}x^6-\frac{5}{16}x^4\la+\frac{5}{16}x^2\la^2+\frac{1}{16}\la^3=\frac{1}{16}(\la-x^2)(\la^2+6x^2\la+x^4)
\]
which has three distinct roots in $(z=0) \cong \mb{P}(1,2)$. 
These include $[0:1:0]$, which is an $F$-rational $\mu_3$-fixed point of $\mc{C}^{\ell}_3$. It follows that $\mc{C}^{\ell,\nu}_3$ has three distinct $\mu_3$-fixed points as well, and applying the Riemann-Hurwitz 
formula over $\ov{F}$ \cite[Thm. 1.10]{O} to the $\mu_3$-quotient $\mc{C}^{\ell,\nu}_3 \rightarrow \mc{Q}^{\ell,\nu}_3$ over $\overline{F}$, we deduce that
\begin{align*}
    3(2g(\mc{Q}^{\ell,\nu}_3)-2)+3(3-1)&=2g(\mc{C}^{\ell,\nu}_3)-2.
\end{align*}
Substituting $g(\mc{C}^{\ell,\nu}_3)=1$ and solving for $g(\mc{Q}^{\ell,\nu}_3)$, we see that $\mc{Q}^{\ell,\nu}_3$ is a smooth rational curve over $\overline{F}$. Moreover, the existence of an $F$-rational point of $\mc{Q}^{\ell,\nu}_3$ induced by $[0:1:0] \in \mc{C}^{\ell}_3$ implies that $\mc{Q}^{\ell,\nu}_3 \cong \mb{P}^1_F$ over $F$ by \cite[Theorem A.4.3.1]{HS}. Indeed, since the only singular points of $\mc{C}^{\ell}_3$ are three distinct nodes that form a single $\mu_3$-orbit, the quotient $\mc{Q}^{\ell}_3$ is $F$-isomorphic to a nodal plane cubic with an $F$-rational smooth point.
\end{proof}

\begin{prop}\label{prop:Cell3_quotient_two}
    Whenever $n=2\ell$ and $F$ is perfect of characteristic 3, the $\mu_3$-quotient $Q^{\ell}_3$ 
    of $\mc{C}^{\ell}_3$ is $F$-isomorphic to a union of two $\mb{P}^1_F$'s glued tacnodally at a single common $F$-rational point, i.e., such that the tangent lines of the two components at the common $F$-rational point are identified.
\end{prop}

\begin{proof}
Since $F$ is perfect, we may assume $F \cong \mb{F}_3$ as above, and $P^{\ell}_3=-x^6+x^4\la-x^2\la^2-x^3+\la^3-x\la-1$. The corresponding curve is indeed smooth away from the point $[1:0:1] \in \mb{P}(1,2,1)$, and is transverse to the line $(z=0)$ at infinity. Since $\mu_3$ is a nonreduced multiplicative group scheme, this situation requires a separate analysis, as surveyed in
\cite[\S~2.2]{Mat} and \cite[\S~3]{Tz}. To do so, we work locally in affine charts. 
Accordingly, let $\mathrm{Spec}\:A:=\mathrm{Spec}\:\mb{F}_3[x,\la]/P^{\ell}_3$. The $\mu_3$-action on $\mathrm{Spec}\:A$ is dual to a coaction morphism $\Phi: \mb{F}_3[x,\la]/P^{\ell}_3 \rightarrow \mb{F}_3[x,\la,t]/(P^{\ell}_3,t^3-1)$ of $\mb{F}_3$-algebras, given by
\[
    x \mapsto tx \text{ and }
    \la \mapsto t^{-1}\la=t^2\la.
\]
Furthermore, $\Phi$ is uniquely determined by a derivation $D$ on $\mathrm{Spec}\:A$ with $D^3=D$, and for every $f \in A$, we have $Df=mf$ if and only if $\Phi(f)=t^mf$; in our particular case, $D(x^i\la^j)=(i+2j)x^i\la^j$, so $Dx=x$, $D\la=-\la$, and $DP^{\ell}_3=0$. It follows that the $\mu_3$-quotient of $\mathrm{Spec}\:A$ is $\mathrm{Spec}\:A^D$, where 
\[
A^D:=\{a \in A \:|\: Da=0\} \cong \mb{F}_3[x^3,x\la,\la^3]/P^{\ell}_3 \cong \mb{F}_3[\al,\be,\ga]/(\al\ga-\be^3, -\al^2+\al\be-\be^2-\al+\ga-\be-1);
\]
note that $\mathrm{Spec}\:A^D$ is indeed an affine open subscheme of $Q^{\ell}_3:=C^{\ell}_3/\mu_3$. A similar analysis shows that $Q^{\ell}_3$ is singular precisely in the point $(\al,\be,\ga)=(1,0,0)$ 
of $\mathrm{Spec}\:A^D$. 

\medskip
We now turn to the singular point of $Q^{\ell}_3$. Working along the affine locus $\mathrm{Spec}\:A^D$, we make a linear change of variables $(\alpha \mapsto \alpha+1, \beta \mapsto \be, \ga \mapsto \ga)$, and then let $\ga=\frac{\be^3}{\al+1}$. Clearing denominators, $\mathrm{Spec}\:A^D$ becomes (presented by) $\mathrm{Spec}\:\mb{F}_3[\al,\be]/(-\al^3-\al^2+\al^2\be+\al\be-\be^2+\be^3 )$, which is singular at the origin $(\al,\be)=(0,0)$. To understand the singularity type 
at the origin, we linearly change coordinates via $u=\be-\al$ and $v=\al+\be$. 
The affine part of $Q^{\ell}_3$ at $(0,0) \in \mb{A}^2_{u,v}$ is cut out by an affine plane cubic $(v(-v+u^2+uv-v^2)=0)$, which is a tacnodal union of $(v=0)$ and $(-v+u^2+uv-v^2=0)$ at $(0,0)$. 
Moreover, each of these components has at least one $\mb{F}_3$-rational point; so each is isomorphic to $\mb{P}^1_{\mb{F}_3}$ by \cite[Theorem A.4.3.1]{HS}. Since $Q^{\ell}_3$ has only one singularity, these components do not intersect away from $(0,0) \in \mb{A}^2_{u,v}$; so 
$Q^{\ell}_3$ is isomorphic to a union of two copies of $\mb{P}^1_{\mb{F}_3}$ glued tacnodally at a common $\mb{F}_3$-rational point.
\end{proof}

\section{Inflectionary curves and surfaces from bielliptic curves of genus two}\label{sec:bielliptic}
%

Let $X$ be a curve of genus 2 defined over a field $F$ with $\mbox{char}(F) \neq 2$. We say that $X$ is {\it bielliptic} if 
it admits a degree two morphism to an elliptic curve defined over the algebraic closure $\ov{F}$. Equivalently, $X$ is bielliptic if it has a (non-hyperelliptic) involution $\sigma$ such that $X/\<\sigma\>$ is an elliptic curve.

\medskip
Now assume $X$ is bielliptic. Let $\tau$ denote the hyperelliptic involution of $X$, let $G:=\Aut (X)$ denote the automorphism group of $X$ over $\ov{F}$, and let $\bar G:=G/\<\tau\>$ denote the reduced automorphism group.
Then the image $\bar \sigma$ of $\sigma$ in $\bar G$ acts faithfully on the set $W$ of Weierstrass points of $X$.  
Given an affine model $y^2=f(x)$ for $X$, we may further assume that $\bar \sigma (x)=-x$ and that $1\in W$, by replacing $x$ by $cx$ for a suitably chosen unit $c \in F^{\ast}$. The set of Weierstrass points is $W=\{ \pm 1, \pm \alpha, \pm \beta \}$ for some $\alpha, \beta \in \mathbb P^1 \setminus \{ 0, \infty, \pm 1\}$, and correspondingly the affine equation of $X$ becomes
\begin{equation}\label{bielliptic-eq-1}
    y^2= (x^2-1)(x^2-\alpha^2)(x^2-\beta^2).
\end{equation}
If we do not fix $x=1$ as a Weierstrass point, we may assume that $W=\{\pm \alpha, \pm \beta, \pm \gamma\}$ for some $\alpha, \beta, \gamma \in F^{\ast}$ and that $X$ has equation $y^2= (x^2-\alpha^2) (x^2-\beta^2) (x^2-\gamma^2)$.
We may further replace $x$ by a suitable scalar multiple $\lambda x$ so that $\alpha^2 \beta^2 \gamma^2 =1$. It then follows that
\begin{equation}\label{bielliptic-eq-2}
y^2 = x^6- s_1 x^4 + s_2 x^2 - 1
\end{equation} 
where $s_1=\alpha^2 + \beta^2 + \gamma^2$ and $s_2=\alpha^2 \beta^2+ \alpha^2\gamma^2 + \beta^2 \gamma^2$. The $x$-discriminant of $f(x)=x^6- s_1 x^4 + s_2 x^2 - 1$ is 
\begin{equation}\label{eq:biell_discriminant}
\Delta_x (s_1, s_2)= 64(-s_1^{2}s_2^{2}+4s_1^{3}+4s_2^{3}-18s_1s_2+27)^{2}.
\end{equation}

\noindent We now turn to inflectionary varieties associated to families of bielliptic curves. The following result is elementary. 

\begin{lemma}\label{biell_infl_poly}
 The inflection polynomial  $P^1_m (x, s_1, s_2)$ associated to the two-dimensional family \eqref{bielliptic-eq-2} 
 is divisible by $x$ whenever $m>1$ is odd. Accordingly, we set
\[
Q_m (x, s_1, s_2):= 
\begin{cases}
  P^1_m(x, s_1, s_2) & \text{if m} \text{ is even; and} \\
  \frac 1 x \cdot  P^1_m(x, s_1, s_2) & \text{if }m\text{ is odd}.
\end{cases}
\]
Then $Q_m(x, s_1, s_2)$ is a polynomial in $x^2$; equivalently, $P^1_m$ is an even (resp. odd) polynomial when $m$ is even (resp. odd).
\end{lemma}

\begin{proof} Lemma~\ref{biell_infl_poly} clearly holds when $m=1$ or $m=2$. Arguing by induction, assume it holds for some particular value of $m$. To show that it holds for $m+1$, it suffices to apply the defining recursion 
for the inflection polynomials $P^1_m$ from Proposition \ref{prop:infl_recursion}. Clearly each of the products $D^1 P^1_m \cdot f$ and $P^1_m \cdot D^1 f$ has the required divisibility and polynomiality properties; so any linear combination of these does as well. 
\end{proof}

\noindent Lemma~\ref{biell_infl_poly} implies, in particular, that the inflectionary surface $X_m$ defined by $Q_m$ is naturally a double cover of an auxiliary surface $Y_m$ in coordinates $y,s_1,s_2$ obtained by setting $y=x^2$.

\begin{rem}
\emph{The complexity of the equation of a genus 2 curve is minimized by the presentation using the coordinates $(s_1, s_2)$; however the ordered pair $(s_1, s_2)$ does not uniquely single out the isomorphism class of a genus 2 curve. Uniqueness up to isomorphism may be achieved by instead parameterizing using the invariants $v=s_1^3+s_2^3$ and $w= s_1 s_2$; see \cite{SV}.}  
\end{rem}

\subsection{Inflectionary curves from special pencils of bielliptic curves}
In what follows, $D_n$ denotes the dihedral group of order $2n$. We assume that our base field $F$ is perfect, with $\text{char}(F) \notin \{2,3\}$. Cardona and Quer \cite{CQ} classified curves of genus 2 with automorphism groups isomorphic to $D_4$ or $D_6$ up to $\ov{F}$-isomorphism.
\subsubsection{Genus 2 curves with $\Aut( X) \cong D_4$}
A genus 2 curve has automorphism group isomorphic to $D_4$ if and only if $w^2-4v^3=0$. Up to $\ov{F}$-isomorphism, such a curve is given by 
\begin{equation}\label{D4_model}
 y^2={x}^5+{x}^{3}+ s x.
\end{equation}

Somewhat abusively, we will refer to a {\it $D_4$ pencil} as any pencil of superelliptic curves cut out by $y^n={x}^5+{x}^{3}+ sx$ where $n \geq 2$ and $s \in F$ is an affine parameter. The Newton polygons of the inflection polynomials $P^{\ell}_m(x,s)$ derived from the corresponding $D_4$ pencils are characterized as follows.

\begin{thm}\label{prop:D4_newton_polygon_one}
Suppose that $u=\frac{\ell}{n}$ is neither an integer multiple of $\frac{1}{3}$ nor of $\frac{1}{5}$, and that $\text{char}(F)$ is either zero or sufficiently positive. Given a positive integer $m \geq 2$, let $\mathcal{C}^{\ell}_m= (P^{\ell}_m(x,s)=0)$ denote the $m$-th inflectionary curve associated to the $D_4$ pencil. Its Newton polygon ${\rm New}(\mc{C}^{\ell}_m)$ is the lattice simplex with vertices $(0,m)$, $(2m,0)$ and $(4m,0)$. 
\end{thm}

\begin{proof}
We adopt the same strategy used in proving Theorems~\ref{generic_support_legendre_newton_polygon} and \ref{support_legendre_newton_polygon_z=2} in the preceding section, predicated on identifying critical coefficients of the universal inflection polynomial $P^{\ell}_m=P^{\ell}_m(x,s,u)$ derived from the $D_4$ pencil. 
The desired result follows from the facts that
\begin{equation}\label{eq:D4_bielliptic_vertices}
[(0,m)]P^{\ell}_m=\frac{1}{m!}(u)_m, [(2m,0)]P^{\ell}_m=\frac{1}{m!}(5u)_m, \text{ and } [(4m,0)]P^{\ell}_m=\frac{1}{m!}(3u)_m
\end{equation}
for every $m \ge 2$; and that $[(i,j)]P^{\ell}_m=0$ for every $(i,j) \notin \text{Conv}((0,m),(2m,0),(4m,0))$. 
Both statements follow easily by induction using the recursion of Proposition~\ref{prop:infl_recursion}, starting with the base case $m=2$. Indeed, the fact that $[(i,j)]P^{\ell}_m=0$ for every $(i,j) \notin \text{Conv}((0,m),(2m,0),(4m,0))$ clear when $m=3$; arguing inductively and assuming the analogous statement holds for some $m \geq 2$, we see that 
\[
\text{Conv}(\text{New}(D^1 P^{\ell}_m) \oplus_M \text{New}(f) \bigcup \text{New}(P^{\ell}_m) \oplus_M \text{New}(D^1 f))= \text{Conv}((0,m+1),(2m+2,0),(4m+4,0))
\]
so the required vanishing of coefficients also holds for $m+1$. We leave the similarly easy inductive verification of the formulae \eqref{eq:D4_bielliptic_vertices} to the reader.
\end{proof}

\begin{figure}[h]       
    \fbox{\includegraphics[width=0.45\linewidth]{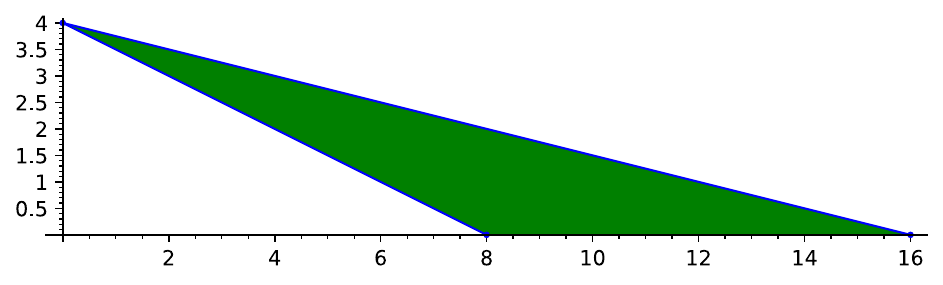}}   
    \hspace{5px}
    \fbox{\includegraphics[width=0.45\linewidth]{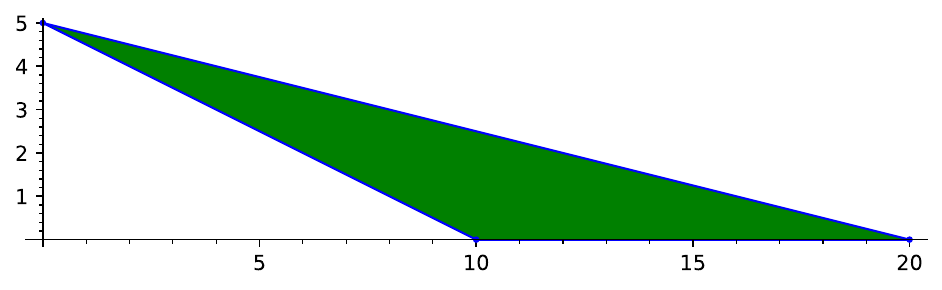}}
    \caption{Newton polygons of the $D_4$-inflectionary curves $\mathcal{C}^{\ell}_m$ for $m=4$ (l) and $m=5$ (r).}
    \label{materialflowChart}
\end{figure}

Note that Proposition~\ref{prop:D4_newton_polygon_one} singles out the weighted projective plane $\mb{P}(1,4,1)$ as a natural choice of ambient toric surface in which to compactify $\mc{C}^{\ell}_m$\footnote{Here the weights 1, 4, and 1 refer to $x$, $s$, and a compactifying variable $z$, respectively.}. Now assume $n=2\ell$. Exactly as in our analysis of Weierstrass inflectionary curves in Section~\ref{sec:legendre_and_weierstrass_pencils}, we anticipate the singularities of $\mc{C}^{\ell}_m$ to be supported in precisely those points corresponding to singular points of the total space of the $D_4$ pencil; i.e., those points whose $s$-coordinates (resp., $x$-coordinates) index singular fibers of the pencil (resp., the $x$-coordinates of their singularities). More precisely, we expect the following to be true.

\begin{conj}\label{conj:D4_inflectionary_singularities}
Assume that $n=2\ell$, and that $\text{char}(F)$ is either zero or sufficiently positive; and let $\mc{C}_m= \mc{C}^{\ell}_m$ denote the $m$-th inflectionary curve derived from the $D_4$ pencil, compactified to a projective curve in $\mb{P}(1,4,1)$. Whenever $m \geq 3$, $\mc{C}_m$ is geometrically irreducible, and has precisely three singularities, whose coordinates in the open affine $xs$-plane are $(0,0)$ and $(\pm \sqrt{\frac{-1}{2}},\frac{1}{4})$. The latter two singularities are permuted by an involution of $\mc{C}_m$; in particular, they are isomorphic.
\end{conj}

In light of Conjecture~\ref{conj:D4_inflectionary_singularities}, it is natural to wonder about the (local) Newton polygons associated to $\mc{C}_m$ in (coordinates centered in) $(\sqrt{\frac{-1}{2}},\frac{1}{4})$.

\begin{conj}\label{conj:D4_newton_polygon_two}
Assume that $n=2\ell$, and that $\text{char}(F)$ is either zero or sufficiently positive; and let $\mc{C}_m$ denote the $m$-th inflectionary curve derived from the $D_4$ pencil $y^n=x^5+x^3+sx$. The Newton polygons $\text{New}_p(\mc{C}_m)$ of $\mc{C}_m$ in $p=(\pm \sqrt{\frac{-1}{2}},\frac{1}{4})$ satisfy
\[
\begin{split}
\text{New}_p(\mc{C}_3)&=\text{Conv}((0,3),(0,2),(2,1),(5,0),(12,0)); \\
\text{New}_p(\mc{C}_4)&=\text{Conv}((0,4),(0,2),(2,1),(8,0),(16,0)); \\
\text{New}_p(\mc{C}_5)&=\text{Conv}((0,5),(0,3),(1,2),(3,1),(9,0), (20,0)); \text{ and }\\
\text{New}_p(\mc{C}_m)&=\text{Conv}((0,m),(\lceil m/2 \rceil, \de_{2|(m-1)}(1,(m-1)/2),(m-2,1), (2m-1,0))
\end{split}
\]
whenever $m \geq 6$.
\end{conj}

Taken together along with Proposition~\ref{prop:D4_newton_polygon_one}, Conjectures~\ref{conj:D4_inflectionary_singularities} and \ref{conj:D4_newton_polygon_two} allow us to produce (a natural expectation for) the geometric genus of $\mc{C}_m$ for every $m \geq 3$.

\begin{conj}\label{D4_geometric_genus}
Assume that $n=2\ell$, and that $\text{char}(F)$ is either zero or sufficiently positive. The $m$-th inflectionary curve $\mc{C}_m$ derived from the $D_4$ pencil $y^n=x^5+x^3+sx$ has geometric genus 0 when $m=3$, and $\lceil \frac{m^2}{2}-m+1 \rceil$ whenever $m \geq 4$.
\end{conj}

Indeed, the arithmetic genus of $\mc{C}_m \sub \mb{P}(1,4,1)$ is computed by the number of interior lattice points of the lattice simplex with side lengths $m$, $m$, and $4m$, which is $\sum_{i=0}^{m-2} (4i+3)= (2m-1)(m-1)$. Assuming $\mc{C}_m$ is irreducible (and reduced), its geometric genus is equal to its arithmetic genus minus the sum of the delta-invariants of its singularities. On the other hand, according to Conjecture~\ref{conj:D4_inflectionary_singularities}, every singularity of $\mc{C}_m$ lies in the open torus of $\mb{P}(1,4,1)$; so its delta-invariant is equal to the number of interior lattice points ``excluded" by the lower hull of the corresponding Newton polygon. It follows that the delta-invariant of the singularity of $\mc{C}_m$ described by Proposition~\ref{prop:D4_newton_polygon_one} is equal to $m(m-1)$. Likewise, the delta-invariant of each of the two isomorphic singularities of $\mc{C}_m$ described by Conjecture~\ref{conj:D4_newton_polygon_two} is equal to 2 when $m=3$; and to $\frac{(m-1)^2}{4}$ (resp., $\frac{m}{2}(\frac{m}{2}-1)$) when $m$ is odd (resp., even) and $m \geq 4$.

\medskip
\noindent Now let $e_{m,q}$ denote the error term $$e_{m,q}=\#\mathcal{C}_m(\mathbb{F}_q)-(q+1)$$ and let $\widetilde{e}_{m,q} = \frac{e_{m,q}}{2g\sqrt{q}}$ denote its renormalized analogue, where $g$ is the geometric genus of $\mc{C}_m$. It is easy to check that $g(\mathcal{C}_2)=1$, i.e., that the desingularization $\mc{C}_2^{\nu}$ of $\mc{C}_2$ is an elliptic curve. Indeed, as a curve in $\mb{A}^2_{x,s}$, $\mc{C}_2$ as an affine curve has defining equation
$3x^4+ 22x^6+ 15x^8+ 6x^2 s+ 30x^4 s- s^2=0$; and its compactification in $\mb{P}^2_{x,s,y}$ has singularities $[0:1:0]$ and $[0:0:1]$. Abusively, we let $\mc{C}_2 \subset \mb{P}^2_{x,s,y}$ denote this projective plane curve. 
Let $t:=\frac{s-15x^4-3x^2}{x^2}$; the equation for $\mc{C}_2$ then implies that $x^4 t^2= x^4(240x^4+112x^2+12)$, from which we conclude that $\mc{C}_2$ is birational to $E_2: t^2=240x^4+112x^2+12$. Here $E_2$ is a smooth curve in $\mb{A}^2_{x,t}$, and defines a double cover of $\mb{P}^1$ branched along the four complex roots of the binary quartic $240x^4+112x^2+12$. Thus $E_2$ is a smooth elliptic curve whenever the characteristic of $F$ is not 2, 3, or 5. Moreover, an easy calculation using \cite[eq. (15)]{AA} (renormalized by a multiplicative factor of 1728) shows that $E_2$ has irrational $j$-invariant, which implies in particular that $E_2=\mc{C}_2^{\nu}$ does not have complex multiplication.


\begin{figure}
\includegraphics[scale=0.6]{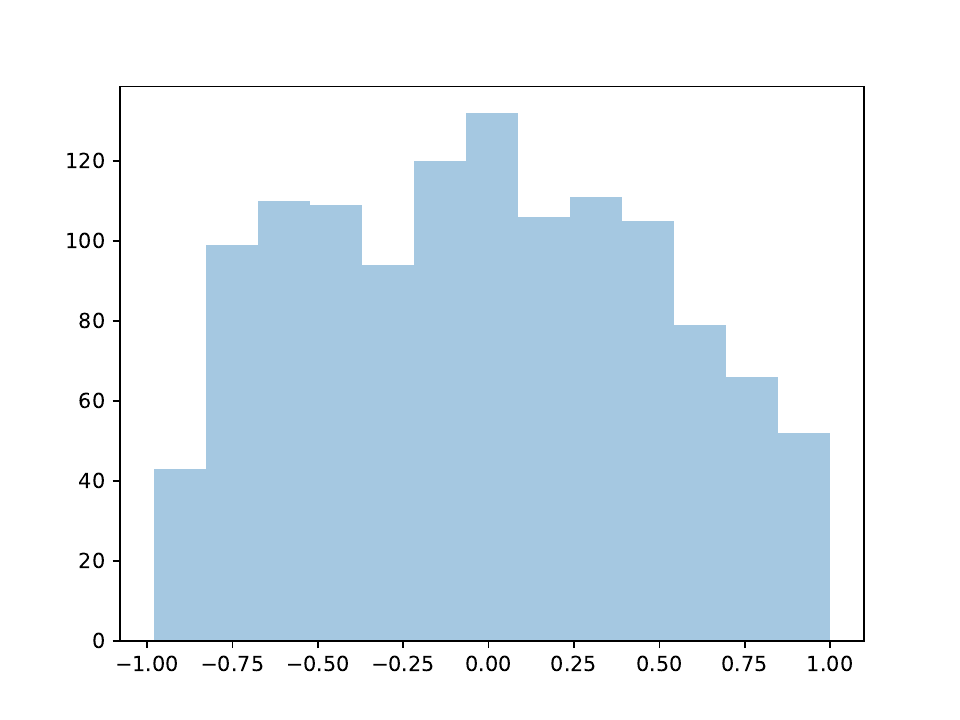}
\caption{Distribution of renormalized errors for the $D_4$ inflectionary curve $\mc{C}_2$ for primes $p\leq 10000$.}
\label{sato-tate_histogram}
\end{figure}

\begin{prop}\label{prop:D4_sato-tate}
The values of the renormalized errors $\widetilde{e}_{2,p}$ are equidistributed with respect to the Sato-Tate measure on an elliptic curve without complex multiplication.
\end{prop}

\begin{proof}

By construction, we have $s=tx^2+15x^4+3x^2$; so the morphism $\nu_{x,t}: E_2 \ra \mc{C}_2$ defined by 
\[
(x,t) \mapsto [x:tx^2+15x^4+3x^2:1]
\]
gives the normalization of $\mc{C}_2$ over the affine chart of $\mb{P}^2_{x,s,y}$ in which $y \neq 0$. The valuative criterion of properness implies that $\varphi$ extends uniquely to a birational morphism, which we also abusively denote by $\nu$, defined over the smooth projective elliptic curve $E_2$. The fiber of $\nu$ over $[0:0:1]$ is defined by the conditions $x=0$, $tx^2+15x^4+3x^2=0$, and $t^2=240x^4+112x^2+12$; that is, $t^2 = 12 = 2^2 \cdot 3$.

\medskip
Similarly, to determine the fiber of $\nu$ over $[0:1:0]$, we use the toric coordinates of the Hirzebruch surface $\mb{F}_2$ in which $E_2$ naturally embeds; see, e.g. \cite[Sec. 2.5]{CDH}. Here $\mb{F}_2$ has toric affine charts $(x,t)$ and $(z,w)$, whose transition maps on overlaps are given by $z=x^{-1}$ and $w=tx^{-2}$, or equivalently $x=z^{-1}$ and $t=wz^{-2}$. Over $\mb{A}^1_z\setminus \{0\} \times \mb{A}^1_w$, the equation of $E_2$ becomes $(wz^{-2})^2= 240 z^{-4}+ 112 z^{-2}+12$; thus $E_2$ is defined in the $(z,w)$-chart by
\[
w^2=240+112z^2+ 12z^4.
\]
Applying the same change of coordinates to $\nu_{x,t}$ yields the normalization $\nu_{z,w}$ in the $(z,w)$ chart; we find
\[
(z,w) \mapsto [z^3:w+15+3z^2:z^4].
\]
It follows that the fiber of $\nu$ above $[0:1:0]$ is the scheme $w^2-240-112z^2-12z^4=0$ and $z^3=0$, whose underlying reduced scheme is defined by $w^2-240-112z^2-12z^4=0$ and $z=0$; that is, $w^2 = 240 = 2^4 \cdot 3 \cdot 5$.

\medskip
Now assume $p>5$ is a prime. The upshot of the preceding two paragraphs is that 
\[
\#(\nu^{-1}[0:0:1](\mb{F}_p))=\left(\dfrac{3}{p}\right)+1 \text{ and } \#(\nu^{-1}[0:1:0](\mb{F}_p))=\left(\dfrac{15}{p}\right)+1
\]
in which we have applied the general fact that that $\mb{F}_p$-rational locus of a scheme is equal to that of its underlying reduced scheme.
As $[0:0:1]$ and $[0:1:0]$ are $\mb{F}_p$-rational points of $\mc{C}_2$, it follows that
\[
\# \mc{C}_2(\mb{F}_p)= \# E_2(\mb{F}_p)- \left(\dfrac{3}{p}\right)- \left(\dfrac{15}{p}\right).
\]
The desired conclusion follows as 
$E_2$ is an elliptic curve without complex multiplication, and passing to the error terms the difference becomes negligible.
\end{proof}

%
%


\subsubsection{Genus 2 curves with $\Aut( X) \cong D_6$}
A genus 2 curve has automorphism group isomorphic to $D_6$ if and only if $4w-v^2+110v-1125=0$. Up to $\ov{F}$-isomorphism, such a curve is given by 
\begin{equation}\label{D6_model}
 y^2={x}^6+{x}^{3} + z.
 \end{equation}

We will refer to a {\it $D_6$ pencil} as any pencil of superelliptic curves cut out by $y^n={x}^6+{x}^{3}+ z$ where $n \geq 2$ and $z \in F$ is an affine parameter.
The dependency on $m$ of those Newton polygons $\text{New}(P^{\ell}_m)$ derived from $D_6$ pencils is more subtle than that of those derived from $D_4$ pencils. Moreover, when $u=\frac{1}{2}$, those inflectionary curves $\mc{C}^{\ell}_m$ that arise from the $D_6$ pencil are always singular in {\it four} distinguished points in the $xz$-plane, namely $(0,0)$ and $(-\frac{1}{2^{1/3}}\zeta^j,\frac{1}{4})$, $j=0,1,2$, whose $z$-coordinates (resp., $x$-coordinates) are roots of the $x$-discriminant $-729z^2(-1+4z)^3$ of $x^6+x^3+z$ (resp., the $x$-coordinates of the corresponding singular fibers of the pencil \eqref{D6_model}). Here $\zeta$ denotes a primitive cube root of unity. On the other hand, it is easy using Proposition~\ref{prop:infl_recursion} to see that for every $m \geq 3$, $(x \mapsto \zeta x, z \mapsto z)$ defines a cyclic automorphism of $\mc{C}^{\ell}_m$ that permutes the three singularities supported in the points $(-\frac{1}{2^{1/3}}\zeta^j,\frac{1}{4})$. Accordingly, it is natural to examine the Newton polygons of inflection polynomials that arise from the $D_6$ pencil in coordinates centered in either the origin or $(-\frac{1}{2^{1/3}},\frac{1}{4})$.

\begin{conj}\label{conj:D_6Newton_polygon}
Suppose that $n=2\ell$, and that $\text{char}(F)$ is either zero or sufficiently positive. Let $\mc{C}_m$ denote the $m$-th inflectionary curve associated to the $D_6$-pencil, and given a point $p \in \mb{A}^2_{x,z}$, let $\text{New}_p(\mc{C}_m)$ denote the Newton polygon of $\mc{C}_m$ associated with affine coordinates centered in $p$. The curve $\mc{C}_3 \sub \mb{A}^2_{x,z}$ is geometrically irreducible, and singular precisely in $p_1=(0,0)$ and $p_{2+j}=(-\frac{1}{2^{1/3}}\zeta^j,\frac{1}{4})$, $j=0,1,2$. Moreover
\[
\text{New}_{p_1}(\mc{C}_3)= \text{Conv}((0,2),(3,2),(6,0),(15,0)) \text{ and } \text{New}_{p_j}(\mc{C}_3)= \text{Conv}((0,2),(2,2),(1,1),(5,0),(15,0))
\]
and $\mc{C}_3$ has geometric genus 2. When $m=4$, the $D_6$ inflection polynomial factors as $P^{\ell}_4=x^2(4z-1) P^{\ell}_{4,*}$, where $P^{\ell}_{4,*}$ defines a curve $\mc{C}_{4,*} \sub \mb{A}^2_{x,z}$ singular precisely in $p_1$, with \[
\text{New}_{p_1}(\mc{C}_{4,*})= \text{Conv}((0,2),(6,0),(12,0))
\]
and of geometric genus 2.\footnote{When $m=4$, the reducible curve $\mc{C}_4$ is in fact singular in $p_j$, $j=2,3,4$; however, those points represent intersections between $\mc{C}_{4,*}$ and the other (geometrically irrelevant) components.} For every $m \geq 5$, the inflection polynomial $P^{\ell}_m$ factors as $P^{\ell}_m=x^{(-m) \text{ mod }3} \cdot (4z-1) P^{\ell}_{m,*}$, where $\mc{C}_{m,*} \sub \mb{A}^2_{x,z}$ cut out by $P^{\ell}_{m,*}$ is irreducible. The affine curve $\mc{C}_{m,*}$ has associated Newton polygons
\[
\begin{split}
\text{New}_{p_1}(\mc{C}_{m,*})&= \text{Conv}(v_1,v_2,v_3,\de_{m \text{ mod }6 \in \{1,2,3\}}v_4,v_5) \text{ and}\\
\text{New}_{p_j}(\mc{C}_{m,*})&= \text{Conv}(v_2,v_3,\de_{m \text{ mod }6 \in \{1,2,3\}}v_4,v_6,v_7,\de_{2|(m-1)}v_8)
\end{split}
\]
where $v_1=(2m- (2m \text{ mod }3),0)$, $v_2=(4m+ 6\lfloor \frac{m-4}{6} \rfloor+ \varphi_1((m-4) \text{ mod }6,0)$, $v_3=(0,\varphi_2(m))$, $v_4=(3,\varphi_2(m))$, 
$v_5=(0,\lfloor \frac{2m}{3} \rfloor)$,
$v_6=(0,\lfloor \frac{m-1}{2} \rfloor)$, $v_7=(m-2,0)$, and $v_8=(1,\frac{m-3}{2})$. Here $\varphi_1(0)=-4$, $\varphi_1(1)=-2$, $\varphi_1(2)=0$, $\varphi_1(3)=-1$, $\varphi_1(4)=1$, and $\varphi_1(5)=3$; while $\varphi_2(3)=2$, $\varphi_2(4)=2$, $\varphi_2(5)=3$, $\varphi_2(6)=4$, and $\varphi_2(m)=4+ \lfloor \frac{m-7}{6}\rfloor \cdot 5+ (m-1)\text{ mod }6$ for every $m \geq 7$. In particular, the delta-invariants $\nu_j^{m,*}$ of $\mc{C}_{m,*}$ in the singularities $p_j$ are given by
\[
\nu_1^{m,*}= 
\frac{3\lfloor 2/3m \rfloor \cdot (\lfloor 2/3m \rfloor-1)}{2} \text{ and } \nu_j^{m,*}= \lfloor \frac{(m-3)^2}{4}\rfloor, j \geq 2
\]
and $\mc{C}_{m,*}$ is of geometric genus
\[
g(\mc{C}_{m,*})= 3(\varphi_2(m))^2- 3\varphi_2(m)-1-\nu_1^{m,*}-3\nu_2^{m,*}+ 3\de_{m \text{ mod }6 \in \{1,2,3\}}(\varphi_2(m)-1)
\]
whenever $m \geq 5$.
\end{conj}

\begin{rem}
\emph{Conjecture~\ref{conj:D_6Newton_polygon}, along with our results for Legendre, Weierstrass, and $D_4$ pencils, suggests that over fields of characteristic zero, there is a tight relationship between singularities of the total space of a pencil of superelliptic curves and singularities of the associated inflectionary curve. The singularities of the total space of a pencil depend, in turn, on the singularity types that arise in fibers. In the case of Legendre and Weierstrass pencils of hyperelliptic curves, any fiber has at-worst a single node and is (geometrically) irreducible; while $D_4$ and $D_6$ bielliptic families include fibers with multiple nodes or simple cusps, and may be reducible. In the final section below we initiate an investigation of inflectionary varieties derived from the full {\it two-dimensional} family \eqref{bielliptic-eq-2} of bielliptic curves, in which case the interaction between the singular loci of the family and of the associated inflectionary surfaces is more subtle.} 

\end{rem}

%
%

%

\subsection{Inflectionary surfaces from bielliptic curves}\label{sec:inflectionary_surfaces_from_bielliptics}
As explained in \cite{Be}, {\it Jung's method} for desingularizing a surface $X$ is in two steps, the first of which is to realize $X$ as a branched cover of a plane $Y$ and compute the embedded desingularization of the discriminant curve of the associated projection $\pi: X \ra Y$.\footnote{In the second step, $X$ is replaced by its blown-up analogue $\wt{X}$ with smooth discriminant; the singularities of $\wt{X}$ are then isolated, and may be resolved via a deterministic combinatorial procedure.} While computing desingularizations of the inflectionary surfaces derived from the superelliptic analogues $y^n= x^6- s_1 x^4+ s_2 x^2-1$ of the bielliptic surface \eqref{bielliptic-eq-2} is itself a natural problem, we will not attack the desingularization problem in full here.\footnote{Kulikov \cite{Ku} has solved the analogue of this problem for two-dimensional families of plane curves subject to a genericity hypothesis.} Rather, we will focus on the structure of the discriminant curves $\De^{\ell}_m$ associated with the projection of the inflectionary surfaces $(P^{\ell}_m=0)$ derived from \eqref{bielliptic-eq-2} to the $(s_1,s_2)$-plane. 
Note that $\De^{\ell}_m$ always contains the discriminant $\De$ of the bielliptic surface \eqref{bielliptic-eq-2}. 

\medskip
We begin by describing the stratification of the (reduced scheme associated with the) discriminant of \eqref{bielliptic-eq-2} according to the singularity configurations along the curves it parameterizes. According to equation~\eqref{eq:biell_discriminant}, the reduced discriminant is a quartic curve $\De_* \sub \mb{A}^2_{s_1,s_2}$. In fact, it is easy to see that $\De_*$ has nodes in the points $(3\zeta^j,3\zeta^{-j})$, where $\zeta$ is a cube root of unity, and that these nodes are permuted by a cyclic $\mu_3$-automorphism of $\De_*$. In particular, $\De_*$ is of geometric genus {\it zero}. We now apply a classical algorithm of Max Noether using adjoint curves (see, e.g., \cite[Ch. 4]{Se}), to compute a parameterization for $\De_*$. More precisely, we single out adjoint conics through the singularities $(3\zeta^j,3\zeta^{-j})$ and the smooth point $(-1,-1)$ of $\De_*$; there is a pencil of these, parameterized by
\[
a_t(s_1,s_2,z)=ts_1^2+ s_1s_2+ (3t-6)s_1 z- (t-2)s_2^2- 3t s_2 z- 9z^2
\]
where $t \in \mb{P}^1$. Solving the system of equations
\[
\text{res}_{s_1}(a_t,\De_*)= \text{res}_{s_2}(a_t,\De_*)= \text{res}_{z}(a_t,\De_*)=0
\]
in which ``res" denotes the resultant, we deduce that the closure of $\De_*$ in $\mb{P}^2_{s_1,s_2,z}$ is parameterized by 
\begin{equation}\label{eq:delta_star_eq}
[s_1(t): s_2(t): z(t))]= [(t-2)(3t^3-6t^2+12t-8):t(3t^3-12t^2+24t-16):t^2(t-2)^2].
\end{equation}
From \eqref{eq:delta_star_eq}, in turn, it is easy to identify those points of $\De_*$ corresponding to curves with singularities locally over $\ov{F}$ of the form $y^n=x^m$ with $m \geq 3$; indeed, these are precisely the solutions of
\begin{equation}\label{eq:delta_star_higher_mult}
f(t,x)= D_x f(t,x)= D^2_x f(t,x)=0
\end{equation}
where $f(t,x)= x^6- \frac{s_1(t)}{z(t)} x^4+ \frac{s_2(t)}{z(t)} x^2-1$. The system \eqref{eq:delta_star_higher_mult} has eight solutions, divided into two groups of four each for $t=1 \pm \sqrt{\frac{-1}{3}}$. It is furthermore clear from the presentation \eqref{bielliptic-eq-1} that these are the {\it only} special configurations possible.

\subsubsection{Further components of the inflectionary discriminant}
To conclude, we describe the components of $\De^{\ell}_m$ for small values of $m$.

\medskip
\noindent {\bf Case: $m=3$.} In this case, the reduced subscheme of the inflectionary discriminant decomposes as $\De^{\ell}_3= \De_* \cup \De^{\ell}_{3,1} \cup \De^{\ell}_{3,2}$, where $\De^{\ell}_{3,1}$ and $\De^{\ell}_{3,2}$ have defining equations $4s_1- s_2^2=0$ and
\[
-78125 - 118125 s_1^3 + 756 s_1^6 + 318750 s_1 s_2 - 31050 s_1^4 s_2 + 
 204375 s_1^2 s_2^2 - 189 s_1^5 s_2^2 - 337500 s_2^3 + 500 s_1^3 s_2^3=0
\]
respectively. Clearly $\De^{\ell}_{3,1}$ is a smooth rational curve. On the other hand, the Newton polygon of $\De^{\ell}_{3,2}$ has 10 interior lattice points, while the closure of $\De^{\ell}_{3,2}$ in the toric surface whose underlying polygon is $\text{New}(\De_*)$ is singular in precisely 9 points, all of which lie in the dense open lous $\mb{A}^2_{s_1,s_2}$. Closer inspection shows that
each of these is a node; so $\De_{3,2}^{\ell}$ is of geometric genus 1.

\medskip
\noindent {\bf Case: $m=4$.} The (reduced) inflectionary discriminant decomposes as $\De^{\ell}_4= \De_* \cup \De^{\ell}_{4,1} \cup \De^{\ell}_{4,2} \cup \De^{\ell}_{4,3} \cup \De^{\ell}_{4,4}$, where $\De^{\ell}_{4,1}$ and $\De^{\ell}_{4,2}$ are smooth rational curves defined by $s_1^2-4s_2=0$ and $s_2^2-4s_1=0$, while $\De^{\ell}_{4,3}$ and $\De^{\ell}_{4,4}$ are defined by
\[
\begin{split}
&-1125+ 4s_1^3+ 110 s_1 s_2- s_1^2 s_2^2+ 4s_2^3=0 \text{ and } \\
&20796875 + 3429000 s_1^3 + 52272 s_1^6 - 13942500 s_1 s_2 - 235440 s_1^4 s_2 - 571350 s_1^2 s_2^2 + 1512 s_1^5 s_2^2 + 3429000 s_2^3\\
&+ 6220 s_1^3 s_2^3 - 235440 s_1 s_2^4 - 3645 s_1^4 s_2^4 + 1512 s_1^2 s_2^5 + 52272 s_2^6=0
\end{split}
\]
respectively. Remarkably, $\De_*$ and $\De^{\ell}_{4,3}$ share the same Newton polygon, namely 
\[
\text{New}(\De^{\ell}_{4,3})=\text{New}(\De_*)=\text{Conv}((0,0), (3,0), (2,2), (0,3)).
\]
In particular, $\De^{\ell}_{4,3}$ is of arithmetic genus 3. On the other hand, closer inspection shows that $\De^{\ell}_{4,3}$ is singular in the points $(-5 \zeta, -5 \zeta^{-1}) \in \mb{A}^2_{s_1,s_2}$, which are permuted by a cyclic $\mu_3$-automorphism of $\De^{\ell}_{4,3}$; in particular, $\De^{\ell}_{4,3}$ is itself a singular rational curve. Likewise, 
we have $\text{New}(\De^{\ell}_{4,4})=2\text{New}(\De_*)$; as $\text{New}(\De^{\ell}_{4,4})$ contains 17 interior lattice points, it follows that the arithmetic genus of the closure of $\De^{\ell}_{4,4}$ in $\text{Tor}(\De_*)$ is 17. On the other hand, (the closure of) $\De^{\ell}_{4,4}$ is singular in 15 points of $\mb{A}^2_{s_1,s_2}$, each of which is a node; so $\De^{\ell}_{4,4}$ is of geometric genus 2.

\medskip
\noindent {\bf Case: $m=5$.} The (reduced) inflectionary discriminant decomposes as $\De^{\ell}_5= \De_* \cup \De^{\ell}_{5,1} \cup \De^{\ell}_{5,2} \cup \De^{\ell}_{5,3}$, where $\De^{\ell}_{5,1}=\De^{\ell}_{4,1}$ and $\De^{\ell}_{5,2}$ is the smooth rational curve defined by $-8+4s_1s_2-s_2^3=0$, while $\De^{\ell}_{5,3}$ is defined by
{\tiny
\[
\begin{split}
&-47148698016885339 - 1856430918636308 s_1^3 + 913227582180384 s_1^6 + 
 23648174414208 s_1^9 + 140560353536 s_1^{12} + 711244800 s_1^{15} \\
 &+ 33089373317849562 s_1 s_2 - 4358733253684512 s_1^4 s_2
 -300895277832256 s_1^7 s_2 - 4020413538816 s_1^{10} s_2 - 3180234240 s_1^{13} s_2
 \\
 &+209782436304845 s_1^2 s_2^2+ 619503583440656 s_1^5 s_2^2 + 
 23057043202656 s_1^8 s_2^2 + 118060510464 s_1^{11} s_2^2 
 -265317120 s_1^{14} s_2^2 - 5091695995399316 s_2^3 \\
 &+ 1335207549416408 s_1^3 s_2^3+ 15852912617504 s_1^6 s_2^3- 519443690880 s_1^9 s_2^3 + 1139913216 s_1^{12} s_2^3 - 
 1226356470049056 s_1 s_2^4 \\
 &- 237101028007542 s_1^4 s_2^4 - 
 2958769613112 s_1^7 s_2^4- 12989609248 s_1^{10} s_2^4 + 17200512 s_1^{13} s_2^4
 + 55678112275728 s_1^2 s_2^5 + 16933654916732 s_1^5 s_2^5\\ 
 &+ 147246610976 s_1^8 s_2^5- 65461824 s_1^11 s_2^5 + 241313946391584 s_2^6- 23675574320096 s_1^3 s_2^6 - 579913339582 s_1^6 s_2^6 - 2298341488 s_1^9 s_2^6 - 
 3304800 s_1^{12} s_2^6\\
 &+ 2280390203328 s_1 s_2^7+ 688503977416 s_1^4 s_2^7 + 
 33223815960 s_1^7 s_2^7- 3479328 s_1^{10} s_2^7 + 3980066847328 s_1^2 s_2^8 - 
 89661598176 s_1^5 s_2^8 - 881911439 s_1^8 s_2^8 \\
 &+ 1118124 s_1^{11} s_2^8- 4300540393088 s_2^9 - 196377843072 s_1^3 s_2^9 + 3550852496 s_1^6 s_2^9
 -425286 s_1^9 s_2^9 + 286292024832 s_1 s_2^{10} + 6808519392 s_1^4 s_2^{10} \\
 &- 30020384 s_1^7 s_2^{10} + 81 s_1^{10}
 s_2^{10}- 30735987456 s_1^2 s_2^{11} - 
 11464256 s_1^5 s_2^{11}- 468 s_1^8 s_2^{11} + 29621700864 s_2^{12}
 + 210366976 s_1^3 s_2^{12} - 1376 s_1^6 s_2^{12}\\
 &+ 227179008 s_1 s_2^{13} + 9600 s_1^4 s_2^{13} + 5376 s_1^2 s_2^{14} - 50176 s_2^{15}=0.
\end{split}
\]
}
\hspace{-4pt}Here $\text{New}(\De^{\ell}_{5,3})=5\text{New}(\De_*)$, and as $\text{New}(\De^{\ell}_{5,3})$ contains 131 interior lattice points, it follows that the arithmetic genus of the closure of $\De^{\ell}_{5,3}$ in $\text{Tor}(\De_*)$ is 131. On the other hand, (the closure of) $\De^{\ell}_{5,3}$ is singular in 105 points of $\mb{A}^2_{s_1,s_2}$, each of which is a node; so $\De^{\ell}_{5,3}$ is of geometric genus 26.


\end{document}